\documentclass[11pt]{amsart}

\usepackage[english]{babel}

\usepackage{amsmath,amsthm, amssymb, bbm,  color}
\usepackage{graphicx}
\usepackage{fullpage}

\newcommand{\R}{{{\mathbb {R}}}}
\newcommand{\C}{{{\mathbb C}}}
\newcommand{\Z}{{\mathbb Z}}

\newcommand{\U}{{\mathbb U}}
\newcommand{\HH}{{\mathbb H}}
\newcommand{\phii}{{\Theta_\phi}} 
\newcommand{\phiiu}{{\Theta_{\phi+u}}} 
\newcommand{\T}{{\mathcal T}}

\newtheorem{Theorem}{Theorem}

\newtheorem {lemma} [Theorem]    {Lemma}

\newtheorem {proposition}[Theorem]    {Proposition}
\newtheorem {theorem}[Theorem]    {Theorem}

\let\oldmarginpar\marginpar
\renewcommand\marginpar[1]{\-\oldmarginpar[\raggedleft\scriptsize #1]%
{\raggedright\scriptsize #1}}

\newcommand {\cal}{\mathcal}
\newcommand {\eps}{\varepsilon}

\begin{document}

\title{Decomposition of Brownian loop-soup clusters}
\author{Wei Qian}
\author{Wendelin Werner}
\address {
Department of Mathematics,
ETH Z\"urich, R\"amistr. 101,
8092 Z\"urich, Switzerland}
\email
{wei.qian@math.ethz.ch}
\email{wendelin.werner@math.ethz.ch}

\begin {abstract}
We study the structure of Brownian loop-soup clusters in two dimensions. 
Among other things, we obtain
the following decomposition of the clusters with critical intensity: When one conditions 
a loop-soup cluster by its outer boundary $\partial$ 
(which is known to be an SLE$_4$-type loop), then the union of all excursions away from $\partial$ by all 
the Brownian loops in the loop-soup that touch  $\partial$ is distributed exactly like the union of all excursions of a Poisson point process of Brownian excursions in the domain enclosed 
by $\partial$. 

A related result that we derive and use is that the couplings of the Gaussian Free Field (GFF) with CLE$_4$ via level-lines (by Miller-Sheffield), of the 
square of the GFF with loop-soups via occupation times (by Le Jan), and of the CLE$_4$ with 
loop-soups via loop-soup clusters (by Sheffield and Werner) 
can be made to coincide. An instrumental role in our proof of this fact is played by Lupu's description of CLE$_4$ as limits of discrete loop-soup clusters.  
\end {abstract}

\maketitle

\tableofcontents

\section{Introduction and background}

One main result of the present paper is a decomposition of Brownian loop-soup clusters in two dimensions and at their critical intensity. Roughly speaking, we shall see that when one 
conditions such a loop-soup cluster by its outer boundary, the union of the excursions away from this outer boundary by all the loops in the loop-soup that touch this boundary is distributed like the trace of a  Poisson point process of 
(therefore loosely speaking independent) Brownian excursions away from this outer boundary. 
This does provide a decomposition of these loop-soups that  sheds some new light on the relations between loop-soups, the Gaussian Free Field and the conformal loop ensemble (CLE) with parameter 4. 
Along the way, we shall in particular show that the three couplings between two of these three random objects that have been derived in earlier work can be made to commute.  
Before stating and discussing our results more precisely, let us first recall a few facts about CLEs and loop-soups:
\medbreak

\noindent 
{\bf Background on CLEs and Brownian loop-soup clusters.}
Recall that a simple CLE (as introduced and studied in \cite {Sh,ShW}) is a random countable collection $\Gamma$ of disjoint simple loops 
that are all contained in the unit disc $\U$. 
The law of $\Gamma$ is invariant under any conformal automorphism from
$\U$ onto itself (see \cite {ShW}) and therefore one can simply define the image of $\Gamma$ under any given conformal map from $\U$ onto some other domain $D$ to be a CLE in $D$.
CLEs are conjectured to be the scaling limit of discrete lattice-based models and they play a quite central role in the theory of two-dimensional conformally invariant structures, 
see for instance \cite {Wln} and the references therein. There is another related family of -- non-simple --  CLEs, where the loops are allowed to be non-simple \cite {Sh,MS3,MSW}, but we will not discuss those in the present paper.

Each simple CLE comes in two closely related variants: 
The non-nested version where no loop in this family is allowed to be surrounded by another one, and the nested CLEs where on the contrary, each given point is almost surely surrounded by infinitely 
many nested loops. Note that these two variants are two essentially equivalent objects (the law of the former can be obtained by the law of the latter and vice-versa: The outermost loops of a nested CLE form a non-nested CLE, and 
 a simple iterative procedure enables to define the distribution of a nested CLE out of the distribution of a simple CLE). 

The laws of CLEs can be characterized by conformal invariance and an additional natural simple condition that is described and discussed in \cite {ShW}. It turns out that there is only a one parameter family of simple CLE distributions (called the CLE$_\kappa$ for $\kappa \in (8/3, 4]$), and that there exist various equivalent ways to construct them: 
\begin {enumerate}
\item
As collections of outer boundaries of outermost clusters in Poissonian collections of Brownian loops in $D$ -- we will recall a few lines below how this construction goes (see \cite {ShW}). 
\item
Via variants of SLE$_\kappa$ processes:  
Conformal loop ensembles are in fact closely related to Schramm's SLE curves \cite {Sch}. Indeed, the loops in a CLE$_\kappa$ are in fact loop variants of SLE$_\kappa$,  see  \cite {Sh,ShW}. This relation can be made precise and enables to 
construct the CLEs via a SLE-based exploration tree  or via a Poisson point process of SLE bubbles (see \cite {Sh,ShW,WW}). 
\item
Via the Gaussian Free Field when $\kappa =4$: CLE$_4$ is also very directly and closely related to the Gaussian Free Field (referred to as GFF in the sequel), see \cite {MS,MS1,MS2,MS3} (or \cite {Wln,ASW} for short surveys). 
One can view the CLE$_4$ as being the family of ``level lines'' that one can deterministically read off from the GFF. 
The other CLEs can also be constructed from a GFF, but the GFF-CLE relation is less canonical in those cases (and in the present paper, we will only discuss aspects of the relation between 
CLE$_4$ and the GFF).
\end {enumerate}

The set of points in the unit disc that are not encircled by any CLE loop is a random fractal carpet with zero Lebesgue measure, and its Hausdorff dimension has been determined explicitly in terms of $\kappa$ (see \cite {SSW,NW}); the dimension is equal to 
is $1 + (2 / \kappa)  + (3 \kappa /32)$, so that the dimension of the CLE$_4$ carpet turns out to be $15/8$. Recall also that the Hausdorff dimension of SLE$_\kappa$ curves (and loops) is equal to $1 + (\kappa/8)$ (see \cite {Beffara,RS}). 

The present paper will mostly focus on some properties of the realization of CLE in term of outer boundaries of Brownian loop-soup clusters. 
Let us first briefly recall this loop-soup construction of CLEs and various other results from \cite {ShW}. 
One starts with a Poissonian collection of Brownian loops in $\U$ -- this is the Brownian loop-soup defined in \cite {LW}. Loosely speaking, Brownian loops appear independently at random in the unit disc, with an intensity given by a constant $c$ times  a very natural measure $\nu$ on (unrooted) Brownian loops.  In such a loop-soup, there will be only finitely many macroscopic loops (say, of diameter greater than any given $\delta$), but infinitely many small ones (of diameter smaller than $\delta$). These Brownian loops are not all disjoint (in fact any given Brownian loop will almost surely intersect infinitely many other loops of the loop-soup). 
The parameter $c$ describes the intensity of the loop-soup: The larger $c$ is, the more loops there are; for instance, a loop-soup with intensity $c=1$ is the union of two independent loop-soups with intensity $c=1/2$.

One then looks at clusters of Brownian loops (where two loops are in the same cluster if there is a finite chain of overlapping loops that allows to join them); when the intensity $c$ is not  too large, more precisely (see  \cite {ShW}) when $c \le 1$, if one uses the normalization of $\nu$ as in \cite {LW}, then there are several (in fact infinitely many) such loop-clusters. The collection of outermost outer boundaries of these clusters is then a collection of simple disjoint and non-nested loops, that turns out to be a non-nested simple CLE$_\kappa$, where 
the relation between the intensity $c \in (0,1]$ of the loop-soup and the $\kappa \in (8/3, 4]$ that describes the CLE$_\kappa$ is $c = c(\kappa)= (3 \kappa -8) (6 - \kappa) / (2 \kappa)$ (we will write $\kappa (c)$ for the inverse function). In particular, the CLE$_4$ that plays a special role in the present paper corresponds to the critical intensity $c=1$ (when $c >1$, all Brownian loops hook up into one dense cluster). 

Let us just mention to provide additional motivation that the CLE$_4$ carpet is conjectured to describe the scaling limit of critical $q$-Potts clusters for $q=4$, and that the combination of various recent results show that CLE$_3$ carpets (corresponding to $c=1/2$) describe the scaling limit of critical Ising clusters (see \cite {MSW} and the references therein). 

\begin{figure}[ht]
\begin{center}
\includegraphics[scale=.7]{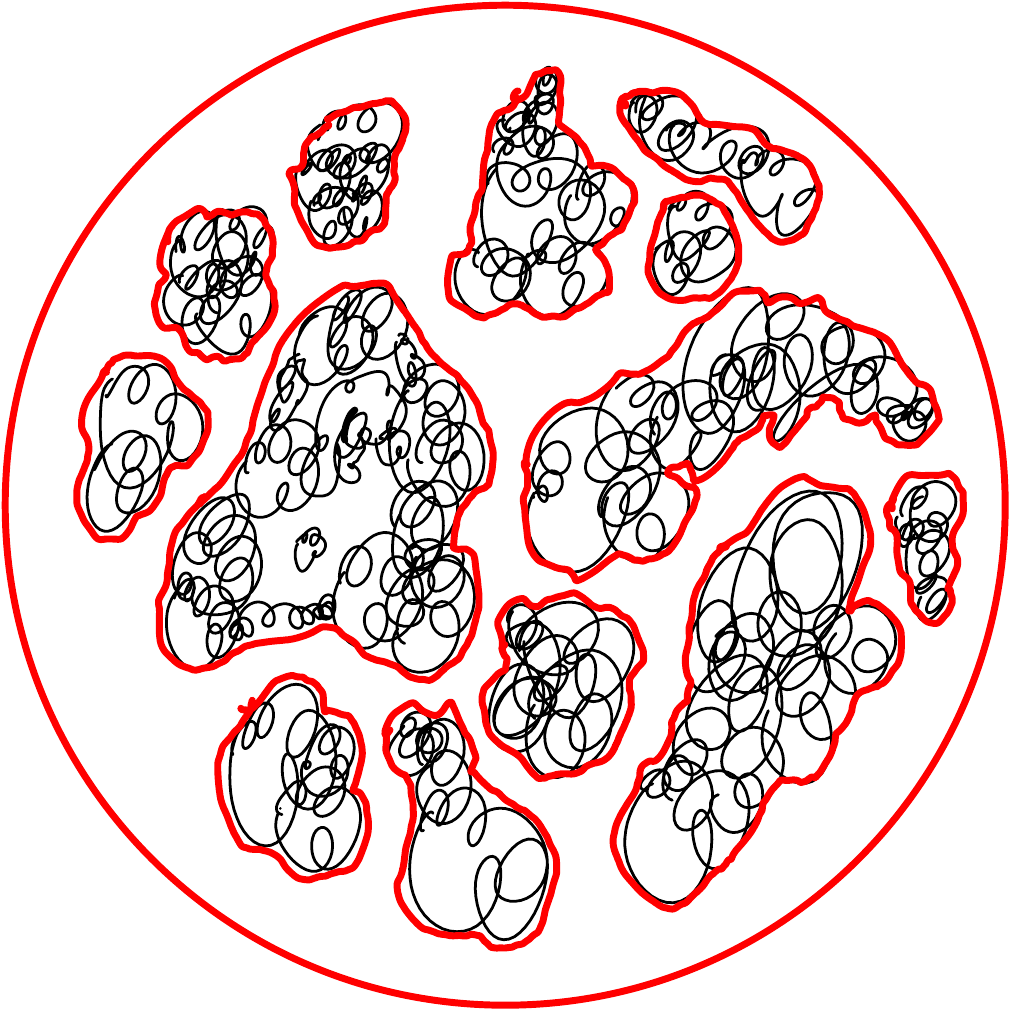}
\end{center}
\caption{\label{fig0}Sketch of a loop-soup. The outermost outer boundaries of clusters of loops form a CLE.}
\end{figure}

\medbreak
\noindent
{\bf Decomposition of the clusters.}
We can note already that there are two types of points on the outer boundary of a loop-soup cluster i.e. of a CLE$_\kappa$ loop.
Those that were part of a Brownian loop (and in fact necessarily on the outer boundary of that Brownian loop), and those that are only reachable via an infinite chain of loops in the loop-soup cluster (i.e. they are in the closure of the cluster but not in the cluster itself). It is not difficult to see that both these types of points do exist: the former  
exists because adding a large single Brownian loop to a given loop-soup creates a configuration that is absolutely continuous with that of a loop-soup (just because of the properties of 
Poisson point processes). And this additional loop may indeed connect two different outermost clusters together, so that a path in the original CLE$_\kappa$ carpet 
that was touching none of the original CLE$_\kappa$ loops and was previously separating two clusters, would now have to hit a point that is both on the additional Brownian loop and on the outer boundary of the newly created cluster. The existence of the latter type of points for instance follows from the fact that the Hausdorff dimension of the outer boundary of the CLE$_\kappa$ loop (which is a SLE$_\kappa$ loop for $\kappa> 8/3$) is strictly larger than the Hausdorff dimension of the outer boundary of one Brownian loop (i.e. of an 
SLE$_{8/3}$ loop). 

Before stating our decomposition result for loop-soup clusters, we first need to briefly recall the definition of the Brownian excursion measure in a simply connected domain $D$. Just as the Brownian loop-measure in $D$ is the  natural (and in two dimensions, conformally invariant) measure on Brownian loops that stay in $D$, the Brownian excursion measure $\mu_D$ is the natural and conformally invariant measure on Brownian paths in $D$ that start and end on $\partial D$ with non-prescribed endpoints. For instance, in the unit disc, one can view it (up to a multiplicative normalizing constant) as the limit when $\eps \to 0$ of $1/\eps$ times the law of a Brownian motion started uniformly on the circle of radius $\exp (-\eps)$ and stopped upon exiting $\U$.  Just as the loop measure, the excursion measure is an infinite and conformally invariant measure, so that one can define it in any simply connected domain $D$
as the conformal image of this measure in $\U$ (the fact that $\partial D$ may be a fractal curve is therefore no problem).  It is also easy to define similarly the excursion measure in finitely connected domains. 

These excursion measures have been used in the context of restriction properties (see \cite {LW2,LW1,LSWr,Wcrrq}). For instance  (see \cite {Wcrrq}), when one uses the appropriate normalization of $\mu$ (we will come back to this normalization question later in this paper) which is the normalization that we will refer to in the sequel, 
a Poisson point process with intensity $\beta \mu$ of excursions in the upper half-plane that start and end on the negative half-line 
will satisfy one-sided restriction with exponent $\beta$ as defined in  \cite {LSWr}.

We are now ready to state 
the following  decomposition of loop-soup clusters for $c\le 1$. Let us already stress that the main point in this theorem is its very last statement.

\begin {theorem}
\label {mainthm}
Consider a Brownian loop-soup $\Lambda$ with intensity $c \in (0,1]$ in the unit disc $\U$, and consider the collection  $\Gamma=(\gamma_j, j \in J)$ of all the outer boundaries $\gamma_j$ of its outermost loop-soup clusters $K_j$. We know from \cite {ShW} that this is a CLE$_\kappa$ for $\kappa = \kappa (c)$.  
Define for all $j$, the interior $O_j$ of the loop $\gamma_j$ to be the bounded connected component of $\C \setminus \gamma_j$.
Then: 
\begin {itemize}
 \item  Conditionally on $\Gamma$, the collections $(\Lambda \cap \overline O_j)$ for $j \in J$ are independent of each other. 
 \item  Furthermore, conditionally on $\Gamma$, for each $j$, the conditional distribution of 
$\Lambda \cap \overline O_j$ in $\overline O_j$ is conformally invariant. In other words, if we define any $\Gamma$-measurable conformal maps $\psi_j$ from $O_j$ onto $\U$, then 
the law of $\psi_j ( \Lambda \cap \overline O_j )$ does in fact not depend on $\Gamma$. 
One can also decompose this family of loops $\Lambda \cap \overline O_j$ into two (conditionally) independent parts: 
 \begin {enumerate}
  \item A Brownian loop-soup with intensity $c$ in $O_j$ (these are the Brownian loops that do not touch the outer boundary $\gamma_j$ of the cluster).
  \item A collection of loops in $\overline O_j$ that do all touch $\gamma_j = \partial O_j$.
 \end {enumerate}
 \item In the special case where $c=1$, conditionally on $\gamma_j$, the union of $\gamma_j$ with the collection of loops in $\overline O_j$ that touch $\gamma_j$, is distributed like 
 the union of $\gamma_j$ with a Poisson point process of Brownian excursions in $O_j$ with intensity $1/4$.
 \end{itemize}
\end {theorem}

\begin{figure}[ht!]
\begin{center}
\includegraphics[scale=1]{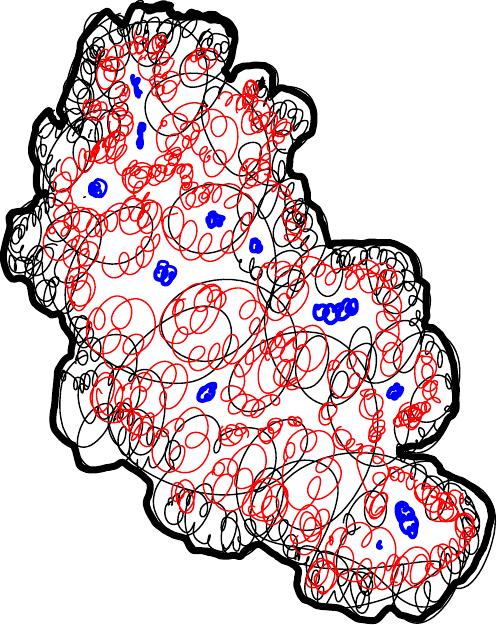}
\includegraphics[scale=1]{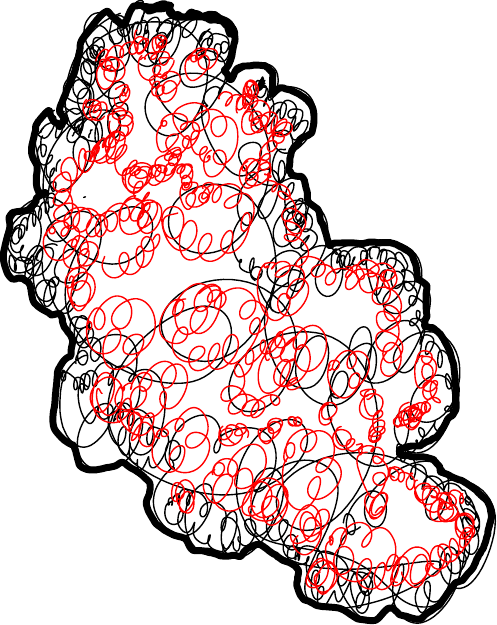}
\includegraphics[scale=1]{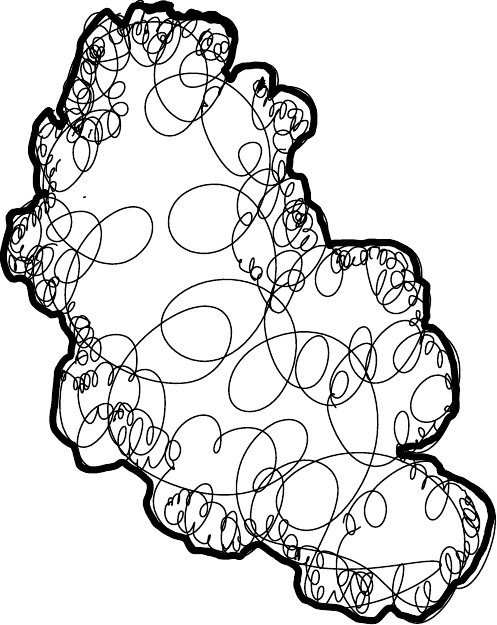}
\end{center}
\caption{\label{figa}Sketch of the whole loop-soup inside an outermost boundary, of the loop-soup cluster itself, and of only the collection of Brownian loops in the cluster that touch its outer boundary.}
\end{figure}

Let us repeat that the main and arguably fairly surprising part of this theorem is its very last one, which is specific to the $c=1$ case.  Given that the excursions away from $\gamma_j$ in the loop-soup do come from  loops, it is somewhat counter-intuitive that they turn out to be related to a Poisson point process of (therefore loosely speaking independently sampled) excursions 
(conditionally on  $\gamma_j$). 

One can reformulate the previous result in the $c=1$ case as a decomposition of the outermost critical clusters $K_j$. Indeed, this decomposition  shows that it is possible to sample the loop-soup clusters in $\U$ as follows: 
\begin {itemize} 
 \item First sample a CLE$_4$ $\Gamma= (\gamma_j, j \in J)$ (this  will turn out to be eventually the family of outer boundaries of outermost clusters of the loop-soup). This family then defines 
 the family of sets $(O_j)$. 
 \item 
 Then, we treat each domain $O_j$ independently, and sample a Poisson point process of excursions with intensity $1/4$ in $O_j$, and independently a loop-soup $\Lambda_j$ in $O_j$. We then define $K_j$ to be the union of $\gamma_j$ with the excursions and all the $\Lambda_j$-loop-soup clusters that touch these excursions. The remaining clusters of the loop-soup (that do not touch 
 any of the excursions) are denoted by $(K_{j,k}, k \in I_j)$. 
\end {itemize}
Then, the obtained family $(K_j, j \in J)$ is distributed exactly like the family of outermost loop-soup clusters of a loop-soup in $\U$, and the families $(K_{j,k}, k \in I_j)$ correspond to the non-outermost loop-soup clusters that are hidden inside of $K_j$.

\medbreak
\noindent
{\bf Commuting GFF/CLE$_4$/Loop-soup couplings.} In the case $c=1$, Theorem \ref {mainthm} and its proof are related to and connect a number of earlier ideas and results: The aforementioned constructions of CLE$_4$ from loop-soups \cite {ShW}, the relation between loop-soups and the square of the GFF \cite {LJ}, the relation between CLE$_4$ and the GFF first pointed out by Miller and Sheffield \cite {MS},
the restriction property type ideas from \cite {LSWr, W, Wcrrq,WW2},  the relation between CLE$_4$ and  the GFF that follows from \cite {LJ} and the
recent results of Titus Lupu \cite {Titus}, and Dynkin's isomorphism theorem \`a la Le Jan-Sznitman (see \cite {Sz} and the reference therein). 
Let us recall briefly the three couplings that have been shown to relate CLE$_4$ with the GFF, the GFF with loop-soups and loop-soups with CLE$_4$ respectively: 
\begin {itemize}
 \item 
Appropriately renormalized cumulative occupation times of Brownian loop-soups define an (appropriately renormalized) squared GFF. This is due to Le Jan \cite {LJ}. Recently, \cite {Titus0}
has also shown (in the discrete counterpart) how to sample the GFF itself (i.e. its sign, given its square) using the loop-soups. This results is a rather direct consequence of the definitions 
of the loop-soups and of the GFF (and it is in fact valid in any dimension, and for any graph). 
\item 
The CLE$_4$ conformal loop ensemble can be coupled with a GFF, in such a way that the CLE$_4$ loops are the ``level lines'' of the GFF. 
This has been first pointed out by Miller and Sheffield \cite {MS}, building on earlier work of Schramm-Sheffield \cite {SchSh2} and Dub\'edat \cite {Dub}
that showed how to couple SLE$_4$ with a GFF. See for instance \cite {Wgff,ASW} for brief reviews of this coupling. This coupling is based on some
natural martingales associated to the SLE$_4$ processes, and that correspond to the progressive discovery of the GFF.  
\item 
As already mentioned, CLE$_4$ loops can be viewed as outer boundaries of loop-soup clusters, as shown in Sheffield-Werner \cite {ShW}. The proof of this fact is related to restriction property ideas.  
\end {itemize}
The proofs of these three couplings have been quite independent, and they used fairly disjoint ideas and techniques. 
One main result of the present paper is that, as schematically shown in Figure \ref {fig3},  these three couplings can be made to commute. This result is in fact instrumental in our derivation of Theorem \ref {mainthm}. 

\begin{figure}[ht]
\begin{center}
\includegraphics[width=11cm]{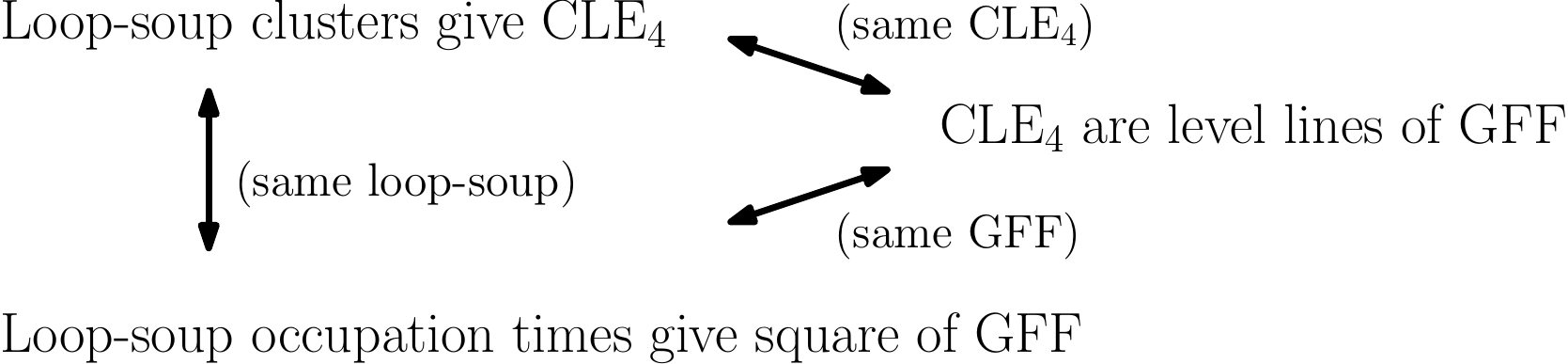}
\end{center}
\caption{\label{fig3}The three couplings can be made to coincide.}
\end{figure}

\medbreak

This  paper is structured as follows: 
We will first derive the first statements of Theorem \ref {mainthm} in Section \ref {S2}. In Section \ref {S3}, we will derive the commutation feature illustrated in Figure \ref {fig3}. 
In this proof, we will make an important use of Lupu's approach to CLE's via loop-soups on cable-systems \cite {Titus0,Titus}. In Section \ref {S4}, we then conclude the proof of Theorem \ref{mainthm},
using this commutation feature fact and Dynkin's isomorphism theorem.  
We then conclude in Section \ref {S5} with various remarks and comments.

\section {Conditioning the loop-soup on the outermost cluster boundaries}
\label {S2}
In this section, we consider $c \in (0,1]$ and all the statements will be valid in this general case.  The value $\kappa (c)  \in (8/3, 4 ]$ is defined as in the introduction. 

Let us consider a loop-soup $\Lambda$ with intensity $c$ in the unit disc, and focus first on the outer boundary $\gamma = \gamma (0)$ of the outermost loop-soup cluster that surrounds the origin.  We denote by $O(\gamma)$ the bounded connected component of the complement of $\gamma$ in the plane, and we define the conformal map $\psi$ from  $O(\gamma)$ into $\U$
such $\psi (0) =0$ and $\psi' (0)$ is a positive real number. Recall (see \cite {ShW}) that $\gamma$ is a continuous self-avoiding loop, so that $\psi$ can be extended to a bijection from the closure of $O (\gamma)$ to the closure $\overline \U$ of the unit disc (that defines a one-to-one correspondence between $\gamma$ and the unit circle). 

Let us define $\Lambda_0$ to be the loop-soup restricted to $\overline {O(\gamma)}$, i.e. the collection of all the loops of $\Lambda$ that are inside $O(\gamma)$ and those that intersect  $\gamma$. 
The collection of loops $\tilde \Lambda := \psi (\Lambda_0)$ is therefore a collection of (Brownian-type) loops in the closed unit disc.

Our first step is to prove the following Lemma: 

\begin {lemma} 
\label {l1}
The loop $\gamma$ and the collection $\tilde \Lambda$ are independent. 
\end {lemma}
In other words, the conditional distribution of $\Lambda_0$ given $\gamma$ is the conformal image of a random independent configuration $\tilde \Lambda$ in the closed unit disc via $\psi^{-1}$. 
This result will be based on  the conformal invariance and the restriction property of the loop-soup, in the spirit of some of the ideas used in \cite {ShW}: 

\begin {proof}
  Let us consider a simply connected subset $U$ of $\U$ that contains the origin. Let us discover all loop-soup clusters that do not fully stay in $U$, consider the 
interior of the complement of the union of all these clusters, and define $U_0$ to be the connected component of this set that does contain the origin and $\psi_0$ the conformal map from $U_0$ onto $\U$ with 
$\psi_0 (0)=0$ and $\psi_0' (0) > 0$. The restriction property of the loop-soup shows immediately that conditionally on $\gamma \subset U$, the law of $\psi_0 ( \Lambda_0)$ is equal to the original (non-conditioned) law of $\Lambda_0$. In particular, this implies that the conditional law of $\tilde \Lambda$ given $\gamma \subset U$ is equal to the unconditional law of $\tilde \Lambda$. This implies readily that 
$\gamma$ and $\tilde \Lambda$ are independent. 
\end {proof}

Actually, the very same argument (just conditioning on all the loop-soup clusters that intersect $U$ instead of conditioning on $\gamma$) shows that  $\tilde \Lambda$ is independent of 
the entire loop-soup in the exterior of $\gamma$, which yields the following result (recall also that the loop-soup is conformally invariant, so that the origin plays no particular role): 
\begin {lemma}
 \label {l2}
 Conditionally on the entire collection  $\Gamma= (\gamma_j, j \in J)$ of outermost boundaries of outermost clusters,
 the families $(\Lambda \cap \overline O_j)$ for $j \in J$ are (conditionally) independent of each other.  
\end {lemma}

Let us now go back to the description of the law of $\Lambda_0$ given $\gamma$. It seems tempting to claim that the family of loops inside of $\gamma$ that do not touch $\gamma$ is distributed like a loop-soup in $O(\gamma)$, because when one discovers $\gamma$ from the outside, one has no information about those loops (this observation has also been pointed out to us by David Wilson \cite {Wilson}). 
This fact turns out to be correct, but one has to be a little careful because of the following caveat: 
Let us call $\Lambda_0^b$ and $\Lambda_0^i$ the collection of loops in $\Lambda_0$ that respectively touch the boundary $\gamma$ and stay in
the open set $O(\gamma)$. Define $\tilde \Lambda^b$ and $\tilde \Lambda^i$ to be their respective image under $\psi$.
Then, it could happen (and as we shall point out towards the end of the paper, this is indeed the case at least for $c=1$)
that the loops of $\Lambda_0^b$ (or equivalently of $\tilde \Lambda^b$) alone do almost surely not form a single cluster (see Fig. \ref {donothook}):  They create a countable collection of disjoint clusters, even if the outer boundary of their closure is equal to $\gamma$. In other words, the loops of $\Lambda_0^b$ ``need'' the contribution of those of $\Lambda_0^i$ in order to form the single cluster that will have $\gamma$ as its outer boundary. 

\begin{figure}[ht!]
\begin{center}
\includegraphics[scale=.6]{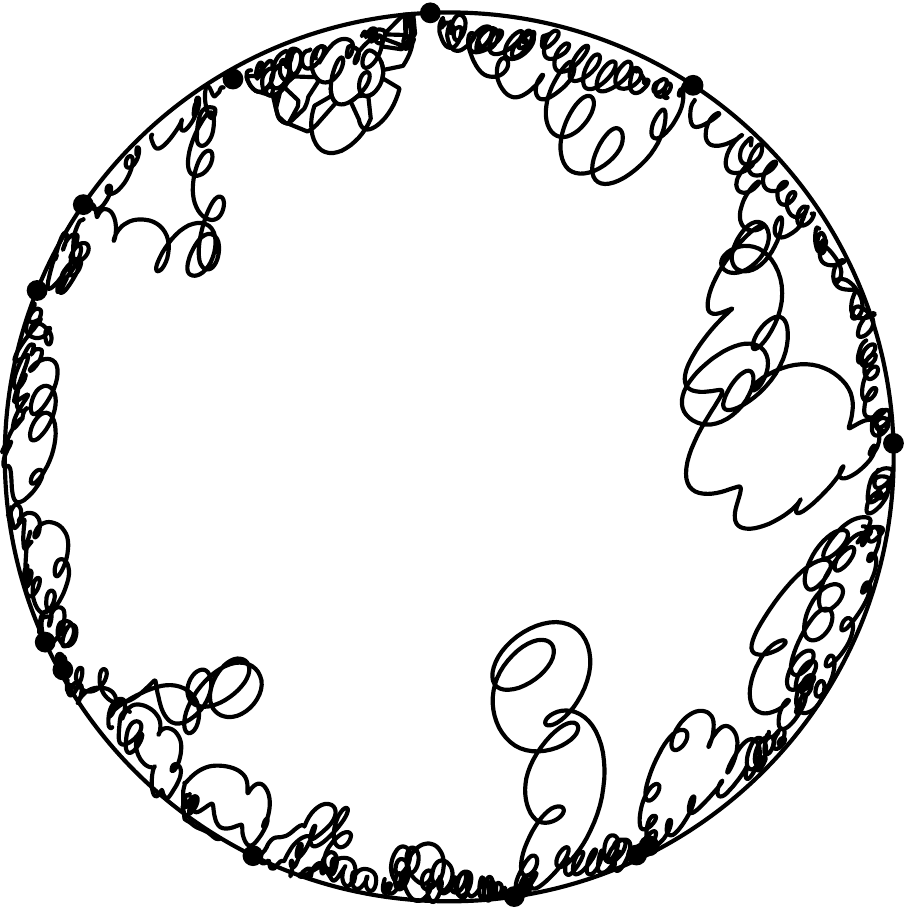}
\caption {\label {donothook} The loops of $\tilde \Lambda^b$ do not hook up into a single cluster (the dots separate clusters)}
\end{center}
\end {figure}

However, as we shall now explain, it turns out that if one considers 
 the union of $\tilde \Lambda^b$ with an independent loop-soup in the unit disc, the obtained configuration does almost surely hook up the loops of $\tilde \Lambda^b$ into a single cluster, which will prove 
 the following statement, for all $c \in (0,1]$:
\begin {lemma}
\label {l3}
The two processes $\tilde \Lambda^i$ and $\tilde \Lambda^b$ are independent. Furthermore, the former is distributed like a Brownian loop-soup in the unit disc.
\end {lemma}

\begin {proof}
 We know a priori that the Brownian loop-soup is locally finite (i.e. for each $\eps>0$, only finitely many loops have diameter greater than $\eps$) and that the same is true for the CLE$_\kappa$ (see \cite {ShW}). 
 It follows that if we decompose the 
 loops of $\Lambda_0^b$ into clusters, then only finitely many of them will reach a distance greater than $\eps$ from $\gamma$ (as each 
 of them would contain at least one Brownian loop with diameter at least $\eps$). 
 
 Let us now consider any deterministic annular region $A \subset \U$ and a set $A_0$ inside the ``middle hole'' of $A$,  such that the distance between $A$ and $A_0$ is greater than $\eps$. We are going to follow the following procedure: We first sample a loop-soup $\Lambda$ in $\U$. 
 Then, we are going to let the Brownian loops in $A_0$ disappear one by one in the order determined by their diameter size: After time $t$, all the loops of $\Lambda$ with diameter greater than $t$ that are in $A_0$ did disappear and we call $\Lambda (t)$ the obtained collection of loops. 
 Note that (because $A_0$ and $t$ are deterministic and the distribution of the number of loops of diameter greater than $t$ in $A_0$ follows a Poisson random variable) the law of $\Lambda(t)$ is absolutely continuous with respect to that of $\Lambda$.

\begin{figure}[ht!]
\begin{center}
\includegraphics[scale=.6]{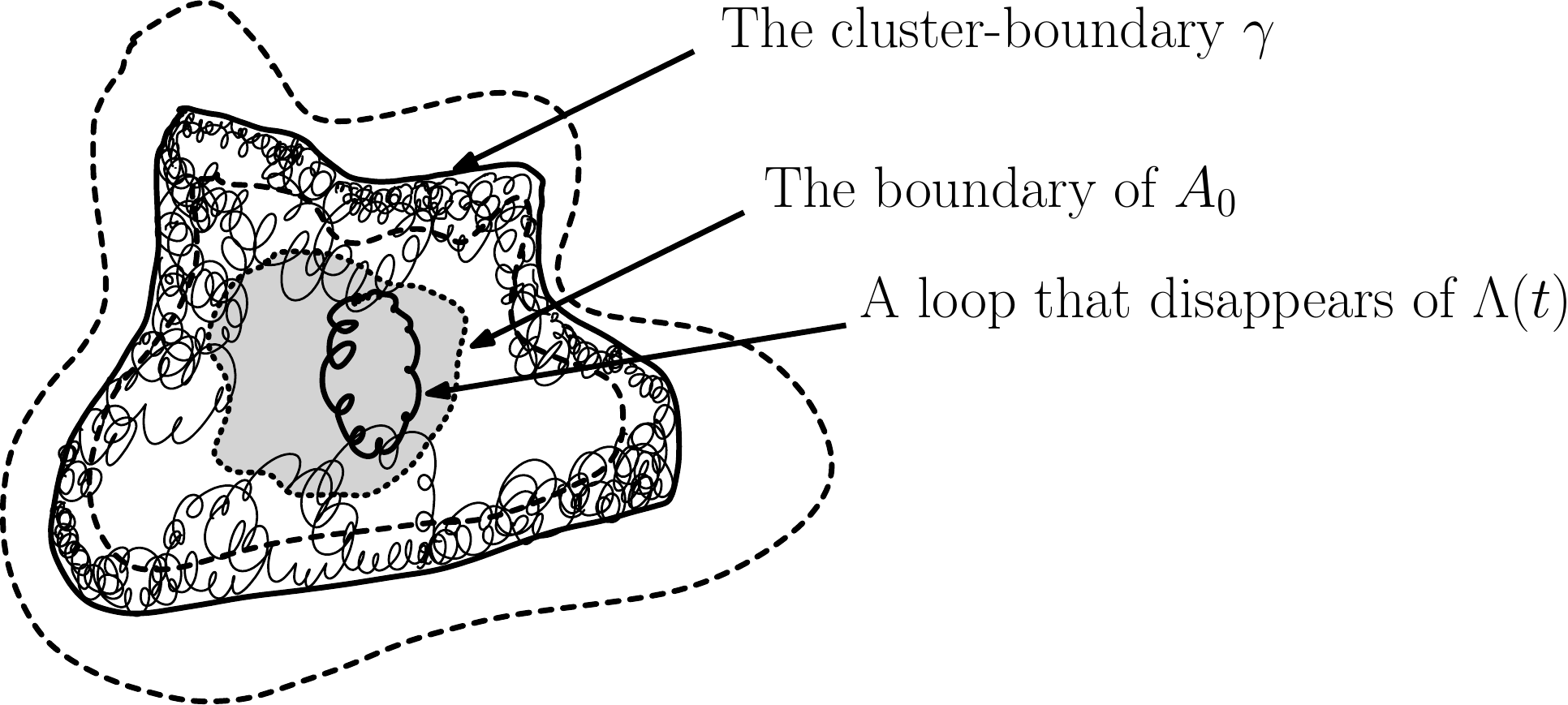}
\caption {The annular region $A$ in-between the dashed loops, the shaded set $A_0$, the loop-soup cluster and its outer boundary $\gamma$.
\label {A0pic}}
\end{center}
\end {figure}
 
Suppose that for $\Lambda$, one observes that $\gamma$ goes around the annular region $A$, like in Figure \ref {A0pic}. Then, we claim that $\gamma$ is still the outer boundary of a loop-soup cluster of $\Lambda (t)$. Indeed, if it wasn't the case, then it means that when removing one of the finitely many loops from $\Lambda$ in order to obtain $\Lambda (t)$, one has disconnected a cluster that had $\gamma$ as its outer boundary into several loop-soup clusters, none of which having the whole of $\gamma$ as its outer boundary. The fact that the loop-soup and the CLE are locally finite implies that it is not possible to split the loop-soup cluster containing $\gamma$ into {\em infinitely} many parts that all touch $\gamma$ by removing just one loop that is at positive distance of $\gamma$. On the other hand, if we would have split the cluster into finitely many parts, because any two loop-soup clusters in $\Lambda(t)$ are at positive distance from each other (recall 
that the law of $\Lambda(t)$ is absolutely continuous with respect to that of $\Lambda$), the intersection of the clusters of $\Lambda(t)$ with $\gamma$ are also at positive distance from 
each other, which leads to a contradiction if several of these clusters do touch $\gamma$. We can therefore conclude that almost surely, for all $t$, there is a loop-soup cluster of $\Lambda (t)$ that has all of $\gamma$ as its outer boundary. 

From this, it follows that on the event where $\gamma$ goes around $A$, resampling the loop-soup in $A_0$ does not change the event that $\gamma$ is the outer boundary of a loop-soup cluster. Since this is true for all deterministic $A$ and $A_0$, we can conclude that for all $\delta$, conditionally on $\gamma$, the law of $\tilde \Lambda^i$ restricted to those loops of diameter at least $\delta$ is that of a loop-soup in $\U$ restricted to those loops of diameter at least $\delta$. Since this law 
is independent of $\gamma$, and that this is true for all $\delta$, we conclude that $\tilde \Lambda^i$ is a Brownian loop-soup in $\U$, which is independent of $\gamma$. This concludes the proof of the lemma.
\end {proof}

This result shows that when observing the trace of all the loops of the loop-soup in an annular region, one can already detect all the outer boundaries of clusters that stay in this annular region (even if the corresponding cluster does not stay in this region). This type of result is reminiscent of the theory of local sets for the GFF developed by Schramm and Sheffield in \cite {SchSh} (the outer boundaries of clusters can be viewed as ``local sets'' of the loop-soup; see \cite {Wresampling} for related simple observations).

\medbreak

We have now decomposed the outermost cluster that surrounds the origin (and all other loops that are contained within it) via three independent inputs: The outer boundary loop $\gamma$, the collection $\tilde \Lambda^b$ and the collection $\tilde \Lambda^i$;  Lemma \ref {l1}, Lemma \ref {l2} and Lemma \ref {l3} imply the first two items in the theorem. It remains to prove the 
last item, which is the description of the  loops that touch $\gamma$ in terms of a Poisson point process of excursions in the special case where $c=1$. This will be the goal of the coming two sections.

\medbreak

Let us conclude this section with  the following remarks: 
Given the CLE$_\kappa$ loops $(\gamma_j)$ consisting of all the outer boundaries of outermost loop-soup clusters, 
the collections  $\Lambda_j := (\Lambda \cap \overline O_j)$ are conditionally independent, and each of their individual law is described as above, so that they in particular contain each a loop-soup $\Lambda_j^i$ in $O_j$. 
This enables to iterate the procedure and to use the clusters of loops of the loop-soup $\Lambda_j^i$ to define a next layer of CLE$_\kappa$ loops inside each $\gamma_j$. In this way, we construct indeed in a deterministic way an entire nested (non-labelled) CLE$_\kappa$ out of a single Brownian loop-soup $\Lambda$ of intensity $c$. But one can note that $\Lambda$ contains strictly more information than this nested CLE$_\kappa$, because quite a lot of information about the loop-soup is  not used in this construction of the nested CLE$_\kappa$: The collection $\tilde \Lambda^b$ is for instance independent of the nested CLE$_\kappa$ constructed in this way. 

\section {CLE$_4$ / GFF / square of GFF / loop-soup couplings} 
\label {S3}

In this section and in the next one, we will suppose that $c=1$, and $\kappa$ will be equal to $4$.
Let us first recall and review  some facts about the coupling between CLE$_4$ (labelled or not labelled), the GFF (and its square) and Brownian loop-soups.

\medbreak

{\sl (a) CLE$_4$ and the GFF.}
Recall that one can define deterministically a CLE$_4$ out of a GFF in $D$ as the family of its level lines. In this way, each loop $\gamma$ of the CLE$_4$ comes equipped with a sign $\eps (\gamma)$ in $\{ +, - \}$, that describes whether this level line is an upward step or downward step i.e. a jump of $2\lambda $ or $-2 \lambda$ where $\lambda := \sqrt { \pi / 8}$, when one moves from the outside of the loop to its inside and one chooses the normalization of the GFF as in \cite {SchSh}. We will use this result in our arguments. This coupling between CLE$_4$ and the GFF, and the fact that the GFF determines the CLE$_4$ and the labels (using the fact that SLE$_4$ is determined by a GFF with appropriate $\pm \lambda$ boundary conditions \cite {SchSh,SchSh2,Dub}, and an absolute continuity argument) is due to 
Miller and Sheffield \cite {MS}. One written proof can be found in \cite {ASW}. 

Recall that conversely, conditionally on the first level CLE$_4$, the labels are i.i.d. and that, given the labelled CLE$_4$, the distribution of the GFF inside each loop $\gamma$ is that of independent GFFs in each loop to which one adds the constants $2\eps (\gamma) \lambda$.   It is also possible to recover deterministically the GFF from the entire nested labelled CLE$_4$. The labelled nested CLE$_4$ and the GFF are therefore two equivalent objects.

Note that there is another nice relation between (unlabelled this time)  CLE$_4$ and the GFF \cite {SWW}, but we will not study it in the present paper.

\medbreak

{\sl (b) The square of the GFF.}
The GFF  $\phi$ being a generalized function, some care is needed when one wants to define its square $\phii$ (usually denoted by $:\! \phi^2 \! :$).
This is however a standard procedure, for instance via the 
language of Gaussian processes and Wick products (see for instance \cite {LJ} and the references therein for background).
Let us make a few simple comments. This (renormalized) squared GFF $\phii $  can be defined as a random generalized function with zero expectation  (i.e., for each smooth function $f$, the 
random variable $\phii  (f)$ has zero mean); it is not a non-negative generalized function even if it is called a square. 
We will give the formula for the characteristic function of $\phii (f)$ in the next paragraph (this gives another description of the law of the process $\phii$).
One concrete way to define $\phii$ is that 
if $B_r(x)$ denotes the ball of radius $r$ around $x$ and $\phi( B_r (x))$ denotes the integral of $\phi$ over this ball, then $\phii(f)$ is the limit in probability as $r \to 0$, of 
$$  \int \frac{1}{\pi r^2}\left( \phi(B_r(x))^2-E[\phi(B_r(x))^2]\right) f(x) dx.$$
This shows in particular that the square of the GFF is a deterministic function of the GFF itself
(mind however that some information about $\phi$ is lost when one just observes $\phii$ and that one cannot deterministically recover the GFF when one knows its square).  
The previous representation of the GFF via a labelled nested CLE$_4$ therefore induces a coupling between the square of the GFF and non-labelled non-nested CLE$_4$ (just keep the first CLE$_4$ layer of the nested labelled CLE$_4$ and forget its label). One can note that (because resampling the sign of the GFF inside each of the CLE$_4$ loops will not change its square) in this coupling of CLE$_4$ with $\phi$ and $\phii $, the non-labelled CLE$_4$ and the square of the GFF are both independent of the labels $\eps (\gamma)$. 
\medbreak

{\sl (c) The square of the GFF and the Brownian loop-soup.}
As pointed out by Le Jan \cite {LJ}, the (renormalized) occupation time measure of the Brownian loop-soup for $c=1$ is distributed exactly as the (renormalized) square of the GFF. 
In order to define the renormalized occupation field of the loop-soup, one can for instance consider the loop-soup $\Lambda(\eps)$ that consists of all loops of $\Lambda$ with time-length at least $\eps$. This is a Poisson point process of Brownian loops, where almost surely, $\Lambda (\eps)$ contains only finitely many loops. Hence, for any non-empty open set $O$ in $\U$, one can define the occupation time $T_\eps (O)$ to be the total time spent in $O$ by all loops of $\Lambda (\eps)$. The corresponding quantity for $\Lambda$ is easily shown to be almost surely infinite (due to the large number of small loops in $O$). However, one can define the limit (in $L^2$, for instance)
$$
T (O) := \lim_{\eps \to 0} [ T_\eps (O) - E ( T_\eps (O) ) ] .
$$
By definition, this limit has zero expectation and it can take negative values. We will call this field $T$ the renormalized occupation field of the loop-soup.
Then (see \cite {LJ}),  $T$ is distributed like $\phii $ (up to a given multiplicative normalizing  constant depending on the chosen normalization for $\phi$). Note that the representation of
 the characteristic function 
of $T (f)$ (when $f$ is a smooth test function with support at positive distance from the boundary of the disc) in terms of the Brownian loop-measure $\nu$ given by
\begin {equation*}
E\bigl( \exp ( i t T (f) ) \bigr) = \exp \bigr( \int  \nu (dl) ( e^{it T(f,l)} - 1 - it T(f,l) ) \bigr)
\end {equation*}
is immediate from the definition of the Poisson point process of loops (here $T(f,l)$ denotes the integral of $f$ along the loop $l$). The renormalization/recentering in the definition of the field $T$ corresponds to introducing the $-it T(f,l)$ term on the right-hand side, that 
ensures that the integral with respect to $\nu$ is convergent (without this term, the contribution of the small loops to this integral diverges).

\medbreak

{\sl (d) CLE$_4$ and the Brownian loop-soup.}
As already discussed in the previous section, it has been proved in \cite {ShW} that the collection of outer boundaries of outermost loop-soup clusters is distributed like a non-nested CLE$_4$. 
Combined with the previous item (c) that shows that this loop-soup defines the square of a GFF, this provides a coupling of $\phii $ with the first layer of a CLE$_4$.

\medbreak

This raises naturally the question whether the couplings between CLE$_4$ and the square of the GFF that are defined via the GFF (by taking the first level lines encountered in the GFF used to define the square of the GFF) and via the Brownian loop-soup (by taking the outermost boundaries of clusters of the loop-soups used to define the square of the GFF) can be made to coincide. The following statement gives a positive answer to this question, and is the main result of the present section:

\begin {proposition}
\label {p1}
One can couple a Brownian loop-soup for $c=1$ with a GFF $\phi$, in such a way that:
\begin {itemize}
 \item The first layer of the CLE$_4$ loops defined by the first level lines of $\phi$ is exactly the outer boundary of the outermost loop-soup clusters.
 \item The renormalized occupation time measure $T$ of the loop-soup is exactly (a constant multiple of)  $\phii$.
\end {itemize}
\end {proposition}

The key to this proposition will be the next lemma, in the spirit of the restriction properties and CLE properties  \cite {LSWr,ShW}: Suppose that we couple a GFF $\phi$ with a Brownian loop-soup, so that the occupation times of the latter define the square of the GFF. Suppose that $A$ is some deterministic compact subset of the closed unit disc, so that $\U \setminus A$ is simply connected. We let $\tilde A$ be the set obtained by removing from $\U \setminus A$ all the insides of loop-soup clusters $\overline O_j$ that do intersect $A$. 
Then, the (easy) restriction property of the loop-soup (see \cite {ShW}) states  that conditionally on $\tilde A$, the law of the loops that stay in $\tilde A$ is exactly a loop-soup in $\tilde A$ (with independent loop-soups in the different connected components of $\tilde A$). 
The following result now provides an analogous feature for the GFF $\phi$:

\begin{figure}[ht!]
\begin{center}
\includegraphics[scale=.6]{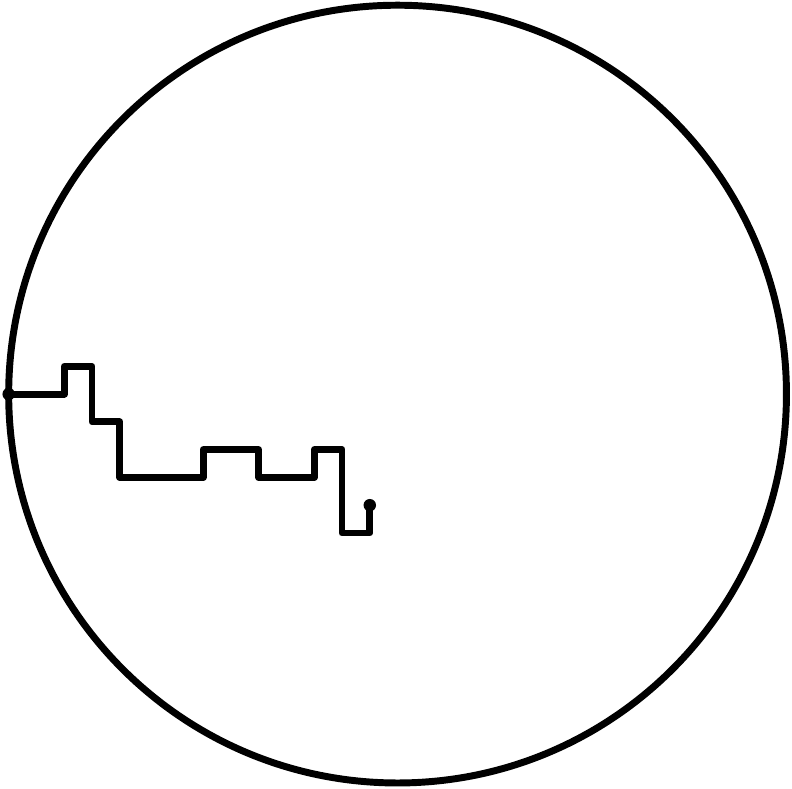}
\includegraphics[scale=.6]{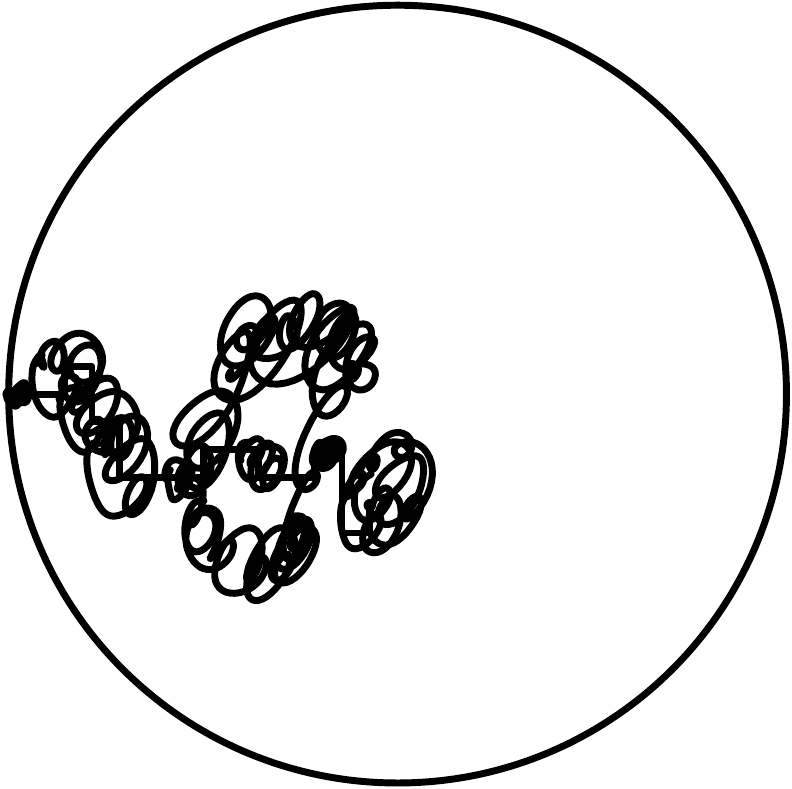}
\includegraphics[scale=.6]{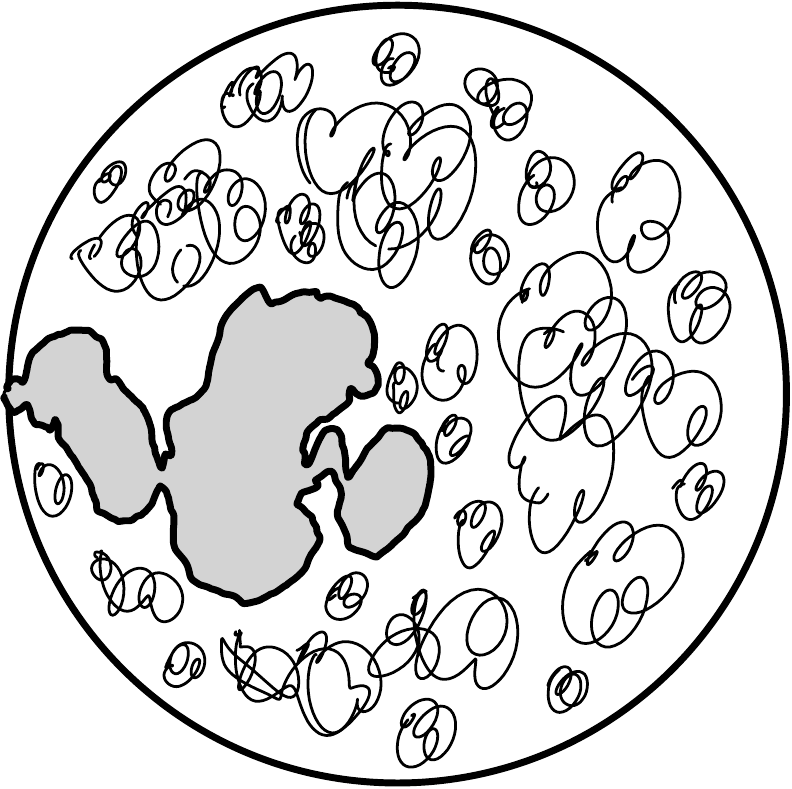}
\caption {The set $A$ (here a union of segments), the set $A$ and the clusters that it intersects, and the loop-soup in the component of $\tilde A$ that contains the origin.}
\end{center}
\end {figure}

\begin {lemma}
\label {lt}
One can couple the GFF and the loop-soup in such a way that $\phii$ is the renormalized occupation time intensity of the loop-soup, and so that for all given $A$, 
conditionally on $\tilde A$, the conditional distribution of the restriction of $\phi$ to $\tilde A$ is just a GFF in $\tilde A$ (with zero boundary conditions).
\end {lemma}

We start with proving the lemma first: 

\begin {proof}[Proof of Lemma \ref {lt}]
Recall that a GFF $\phi$ in $\U$ has a version such that the mapping $f \mapsto \phi (f)$ is continuous on the space of smooth test functions with compact support in $\U$. In particular, in 
order to show that the conditional distribution of the restriction of $\phi$ to $\tilde A$ is a GFF in $\tilde A$, it suffices to see that for a given well-chosen sequence 
$(f_1, \ldots, f_j, \ldots)$ of such test functions (independent of the choice of $A$), for each $j_1, \ldots, j_k$, on the event that the supports of $f_{j_1}, \ldots, f_{j_k}$ are in $\tilde A$, the conditional distribution of 
$(\phi (f_{j_1}), \ldots, \phi (f_{j_k}))$ is the one corresponding to a GFF in $\tilde A$ (i.e. it is that of a Gaussian vector with the appropriate covariance structure).

Our proof will use Titus Lupu's recent approach to continuous structures and CLEs via loop-soups on cable-systems \cite {Titus}: 
Recall that the cable-system on $\delta \Z^2 \cap \U$ is the intersection with $\U$ of the union of all the closed segments of length $\delta$ that join neighbouring points in $\delta \Z^2$. 
On this union of one-dimensional segments, one can naturally define Brownian loops, that correspond to
one-dimensional Brownian motions moving along the segments that constitute the edges of the graph. The trace on the sites of $\delta\Z^2$ of these loops correspond to random walk loops. 

Let us first note that the analogous statement to Lemma \ref {lt} on cable-systems holds. Indeed (see \cite {Titus0,Titus}), suppose that we are looking at the cable-system approximation of $\U$ and $A$ on the square grid with mesh-size $\delta$ and  denote them by $\U^\delta$ and $A^\delta$. Let $\mathcal{L}^\delta$ be the collection of loops of a loop-soup (with appropriate density, that gives rise to the discrete GFF) on this cable-system as defined in \cite {Titus0}. 
Then, one can first discover the loop-soup clusters of this loop-soup that intersect $A^\delta$, and then, exactly as in the continuous case, the part of the 
loop-soup in the complement $\tilde A^\delta$ of the discovered set is (conditionally on $\tilde A^\delta$) distributed exactly like 
a (cable-system) loop-soup in $\tilde A^\delta$. 
Furthermore, it is known that the occupation times of the cable system loop-soup is exactly the square of a GFF on the cable system, and Lupu \cite {Titus} explains how to actually obtain a 
GFF $\phi^\delta$ itself (and not just its square): Just take the square root of the occupation times (i.e. of the local times) of the loop-soup on the cables, and choose the sign of  $\phi^\delta$ independently for each loop-soup cluster. In particular, if one conditions on $\tilde A^\delta$, one does clearly not change the rule to construct the GFF out of the loop-soup in $\tilde A^\delta$,
so that indeed, conditionally on $\tilde A^\delta$, the conditional distribution of the restriction of $\phi^\delta$ to $\tilde A^\delta$ is that of a GFF in $\tilde A^\delta$. 

The idea of the proof is to deduce Lemma \ref {lt} by taking the $\delta \to 0$ limit of this result on the cable systems. One just has to make sure that all the elements of the discrete 
picture do  converge in some appropriate way to their continuous counterparts: 

Let us start with the convergence of the loop-soups themselves (we will then turn to their occupation times and to the loop-soup clusters):
Here, it is easier to describe the convergence of rooted loops, because we will want to control also their 
occupation times. Recall that a rooted loop can be obtained from an unrooted one by just sampling the root uniformly at random (with respect to the time-parametrization of the loop). 
When $\gamma$ and $\gamma'$ are two rooted continuous-time loops in $\U$ (they could actually be loops on a  $\delta$-cable system) with respective time-lengths $t$ and $t'$, we
can use the distance 
$$ d ( \gamma, \gamma' ) := | t-t'| + \sup_{[0,1]} | \gamma (\cdot / t ) - \gamma' (\cdot / t') |.$$
In each $\mathcal{L}^\delta$, if we look at the collection of loops with time-length  greater than $2^{-n}$ and denote it by $\mathcal{L}^\delta_n$, then this collection is a.s. finite, and it converges in law to the collection $ {\mathcal L}_n$ of loops of time-length greater than $2^{-n}$ in a Brownian loop soup, when one uses 
for instance the following distance between finite collections of rooted loops:
$$
d^*(\mathcal{L},\mathcal{L'})= 
\min_{\sigma} ( \max_{\gamma\in\mathcal{L}} d(\gamma,\sigma(\gamma))) 
$$
where the min is taken over all bijections $\sigma$ from 
${\mathcal L}$ to $\mathcal {L}'$, with the convention that $\min \emptyset = \infty$ (so that the distance between two collections of loops that do not have the same number of loops is infinite). 
This follows easily from the weak convergence of the discrete loop measure to the continuous one (see for instance \cite {LTF}). 

The goal of the following few paragraphs is to explain that the recentered occupation time fields of the cable-system loop-soup $\mathcal L^{\delta}$ can be made to converge to the 
renormalized occupation time field of $\mathcal L$.  
Each loop-soup $\mathcal{L}^\delta$ defines an occupation time field $T^\delta $  on 
the cable system, with intensity with respect to the Lebesgue measure on the cable system given by the square of the GFF $\phi^\delta$ on the cable system (see \cite {Titus0,Titus}). Mind that $\phi^\delta$ is a finite continuous function on the cable system and that $T^\delta$ is  non-negative field (it is not the recentered occupation time field). 
For each integer $n$, we denote by $T^\delta_n$ and $T_n$ the (non-recentered) occupation time fields of the set of loops of ${\mathcal L}^\delta_n$ (resp. ${\mathcal L}_n$) with time-length at least  $2^{-n}$.
Due to the previously described convergence of  $\mathcal{L}^\delta_n$,  it is clear that for each $n$, the field $T^\delta_n$ converges in law to  $T_n$ as $\delta \to 0$ (in the 
sense that for any finite set of smooth test functions $(f_1, \ldots, f_j)$ the vector $(T_n^\delta (f_1), \ldots, T_n^\delta (f_j))$ converges in law to the corresponding vector for $T_n$). 
It is also easy to check that for each $n$, $E ( ( T_n^\delta (1) )^2 )$ is bounded independently of $\delta$ (note that  $T_n^\delta (1)$ is the sum of the time-lengths of all 
the loops in $\mathcal {L}^{\delta}_n$).  

In view of studying the loop-soup clusters, it is useful to think in terms of almost sure convergence. 
By Skorokhod's representation theorem, when $(\delta_k)$ is a decreasing sequence that converges to $0$, 
one can find a probability space on which all the loop-soups ${\mathcal L}^{\delta_k}$ and ${\mathcal L}$ 
are simultaneously defined, in such a way that almost surely, for all $n \in \Z$, the collection ${\mathcal L}^{\delta_k}_n$ converges to ${\mathcal L}_n$ and the renormalized occupation times $T^{\delta_k}_n$ converge to  $T_n$ (in the sense that when $(f_j)$ is a given sequence of smooth test functions, these fields applied to $f_j$ do converge). 
By the uniform bound in $L^2$ mentioned at the end of the previous paragraph, we get that for each given $j$,
$E ( T_n^{\delta_k} (f_j) )$ converges to $ E ( T_n (f_j) )$, so that almost surely,
$$ T_n (f_j) - E ( T_n (f_j) ) = \lim_{k \to \infty} [ T_n^{\delta_k} (f_j)-  E ( T_n^{\delta_k} (f_j) )] .$$

On the other hand, 
for each smooth test function $f_j$, we know that
$$ T(f_j) = \lim_{n \to \infty}   [ T_n (f_j) - E (T_n (f_j)) ] $$ 
in $L^2$ (this is just the definition of the renormalized occupation time $T$), and 
that for each given $k$, 
$$ T^{\delta_k} (f_j) - E ( T^{\delta_k} (f_j) ) = \lim_{n \to \infty}  [ T_n^{\delta_k} (f_j) - E ( T_n^{\delta_k} (f_j) ) ]$$
almost surely (this is just because the sum of occupations times of the ${\mathcal L}^{\delta_k}$ and the sum of the expectations both converge). 

We now wish to use $L^2$ bounds in order to interchange the $k \to \infty$ and $n \to \infty$ limits and to conclude that $T(f_j)$ is the limit in probability of  
$ [ T^{\delta_k} (f_j) - E ( T^{\delta_k} (f_j) ) ]$.
For this, let us now first note that for each $j$, the variable $T^{\delta_k} (f_j) - T_n^{\delta_k} (f_j)$ comes from the Poisson point process of loops with time-length at most $2^{-n}$ on the 
cable system. 
The variance of this random variable is therefore decreasing in $n$ and 
equal to the expectation of the sum over all loops of ${\mathcal L}^{\delta_k} \setminus {\mathcal L}^{\delta_k}_n$
of the square of the integral with respect to time of $f_j$ over the loop. 
This integral is bounded by $\sup |f_j|$ times the time-length of the loop. Hence, the variance of  $T^{\delta_k} (f_j) - T_n^{\delta_k} (f_j)$ is bounded 
by $\sup |f_j|^2$ times the integral $I ( \delta_k, n)$ with respect to the loop-measure on the cable-system in the disc, restricted to loops of time-length at 
most $2^{-n}$, of the square of the time-length of the loop. It remains to show that this last integral is bounded independently of $k$ by a quantity that goes to $0$ as $n \to \infty$. 

Let us prove this in two steps. Let us first consider the loop-soups in the square $[-1,1]^2$ instead of
the unit disc, and choose to work with the sequence $\delta_k = 4^{-k}$. The first goal is to prove that the integral $J( \delta_k)$ 
with respect to the loop-measure on the cable-system in this square of the time-length of the loop to the power $7/4$ is bounded independently of $k$ (here $7/4$ is just chosen 
because it is smaller than $2$ but large enough so that the following argument works). 
This is a direct consequence of the fact that 
$$ J (\delta_{k+1}) \le 25 \times 16^{-7/4} J( \delta_{k})  + C $$  for some absolute constant C (cover the square with
$5 \times 5$ squares of side-length $1/4$, use scaling for the contribution to $J( \delta_{k+1})$ for those loops that stay
in one of the 25 squares and use the fact that the loops that are not contained in any one of the 25
smaller squares have a diameter at least $1/16$ -- we leave the details to the reader).
Then, we can note (using the fact that $t^2 \le 2^{-n/4} t^{7/4}$ when $t \in (0, 2^{-n})$) that
$$  I (\delta_k, n ) \le 2^{-n/4} J ( \delta_k)$$
to conclude. 
 
Putting all the pieces together shows that  
$ T(f_j)$ is the limit in probability of  
$ [ T^{\delta_k} (f_j) - E ( T^{\delta_k} (f_j) ) ]$
as $k \to \infty$.
Hence, one can extract a deterministic subsequence of $\delta_k$, such that this convergence takes place 
almost surely along that subsequence. Since this is true for each given $j$ and each sequence $\delta_k$, by the standard diagonal argument, we conclude that one can extract a 
subsequence of $\delta_k$ such that the convergence holds almost surely for all $j$ simultaneously along that subsequence. 
In this way, we have  obtained a joint convergence of the loop-soups and their renormalized occupation time fields for some deterministic sequence of mesh-sizes that goes to $0$. 
In the following paragraphs, we will just again call this sequence $\delta_k$. 

Now, a key feature is the convergence of the cable system loop-soup clusters to the continuous clusters when $\delta \to 0$, established by Lupu in \cite{Titus}. 
Indeed, we know that two Brownian loops in ${\mathcal L}$ that intersect will 
correspond (when $k$ is large enough) in ${\mathcal L}^{\delta_k}$ to loops that intersect as well. Similarly, for all given $A$, the loops of $\mathcal L$ that intersect $A$ will correspond to loops on the cable-system that 
intersect $A$ as well. Hence, when $k \to \infty$, the set $\tilde A^{\delta_k}$ will almost surely be contained in a set (corresponding to the complement of the cluster of macroscopic loops attached to $A$) that converges to $\tilde A$ (in an appropriate topology). Note that the limit of the discrete loop-soup clusters could a priori be larger than the continuous cluster, because of the presence of all the little discrete loops of microscopic or intermediate size (for instance, the collection of 
loops of size smaller than $\sqrt {\delta}$ could percolate). But this is precisely ruled out by Lupu's result: He showed in particular that $\tilde A^{\delta_k}$ converges in distribution to $\tilde A$, so that in the present coupling, one has almost sure convergence of the connected components $\tilde A^{\delta_k}$ to $\tilde A$ (for instance in the Carath\'eodory topology).

Summarizing things, we have the almost sure convergence of ${\mathcal L}^{\delta_k}$, of   $T^{\delta_k}(f_j)- E ( T^{\delta_k} (f_j))$ for all $j$, and  
for each given $A$, the almost sure convergence of the sets $\tilde A^{\delta_k}$ 
to $\tilde A$. We now need to put the GFF itself into the picture (and not just its square).
Each $\mathcal{L}^{\delta_k}$ defines a squared GFF $T^{\delta_k}$ (via its occupation-time field) and, as explained above, it can be coupled with an actual GFF $\phi^{\delta_k}$ on the cable-system by sampling independently the signs of the GFF for each loop-soup cluster. Note that given the convergence of the discrete clusters to the continuous ones,
it is possible to couple these choices of signs for all $k$, in such a way that for any given loop in $\mathcal L$, the sign that will be assigned to the corresponding loop in ${\mathcal L}^{\delta_k}$ will be the same for all $k$ large enough but this will actually not be needed in the argument that follows, because we will again work with convergence in distribution rather than almost sure convergence. 

Let us study the distribution of the triple $({\mathcal L}^{\delta_k}, T^{\delta_k}- E ( T^{\delta_k}), \phi^{\delta_k})$ as $k \to \infty$.   
For each of our given test functions $f_j$, the sequence $\phi^{\delta_k} (f_j)$ is a sequence of Gaussian vectors that converges in distribution to a Gaussian random variable. We can then 
invoke compactness (and a diagonal argument), to deduce that (possibly replacing $(\delta_k)$ by a deterministic subsequence) the triple 
$({\mathcal L}^{\delta_k}, T^{\delta_k}- E ( T^{\delta_k}), \phi^{\delta_k})$ does converge in distribution to the law of some triple $( {\mathcal L}, T, \phi )$ where $\phi$ is a GFF in $\U$ coupled 
in some way to $( {\mathcal L}, T)$. 

The next paragraphs will be devoted to the proof of the fact that is that in this coupling, $T$ is indeed the renormalized square $\phii$ of $\phi$. In other words, we want to check the stand-alone fact that the joint law of a GFF and its renormalized square 
on the cable systems converge as the mesh of the lattice goes to $0$, to the joint law of $( \phi, \phii)$ (indeed, we know that for each $k$,  $T^{\delta_k}$ is exactly the square of the cable system GFF $\phi^{\delta_k}$). 
A first remark is that the trace on the sites of $(\delta \Z^2) \cap \U$ of GFF (and its square) on the cable system is 
exactly the discrete GFF on this graph (and its square). When applied to smooth test functions, the difference between the two vanishes as $\delta \to 0$, so that it is sufficient 
to study the convergence of the discrete GFF (on $\delta \Z^2 \cap \U$) and its square, instead of that of the GFF on the cable systems.

Let us now consider a given smooth test function $f$, and evaluate the $L^2$ norm of the random variable 
$$ Y(r, \delta, f) :=  \int (\pi r^2)^{-2} \left( \phi^{\delta} (B_r(x))^2-E[\phi^{\delta}(B_r(x))^2]\right) f(x) dx -  [ T^{\delta} (f) - E ( T^{\delta} (f) ) ],$$
where we view $\phi^{\delta}$ as a function that is constant on the $\delta \times \delta$ square centered on a site of $\delta \Z^2 \cap \U$.
With this notation,  
$$  T^{\delta} (f) - E ( T^{\delta} (f) )  = \int f(x) [   \phi^{\delta} (x)^2 - E ( \phi^{\delta} (x)^2 ) ]  dx,$$
so that 
$$E [Y(r, \delta, f)^2] = 
E \left[ 
\int \int f(x) f(y) ( U^{\delta,r} (x) - \hat U^{\delta} (x) ) ( U^{\delta,r} (y) - \hat U^{\delta} (y) ) dx dy \right], 
$$
where 
$$ U^{\delta,r} (x):= (\pi r^2)^{-2} \left( \phi^{\delta} (B_r(x))^2-E[\phi^{\delta}(B_r(x))^2]\right) 
\hbox  { and } 
\hat U^{\delta} (x) :=  \phi^{\delta} (x)^2 - E ( \phi^{\delta} (x)^2 ) .$$
Using Fubini and the covariance structure of $\phi^{\delta}$ (using for example Gaussian integration by part), it follows that 
$$
E [Y(r, \delta, f)^2] = 2  \int \int f(x) f(y) (  G_\delta^{r,r}(x,y)^2  + 
  G_\delta(x,y)^2  - 2
 G_\delta^r(x,y)^2 )  dx dy
$$
where $G_\delta (x, y)$ is equal to the discrete Green's function evaluated at the sites of $\delta \Z^2$ closest to $x$ and $y$, and $G_\delta^{r,r}, G_\delta^r$ denote its two meaned out versions
$$G_\delta^{r,r} (x,y) := \int_{B(x,r) \times B(y,r)} G_\delta (x', y') dx' dy' / (\pi r^2)^2, \quad G_\delta^r(x,y):=\int_{B(x,r)}G_\delta(x',y)dx'/(\pi r^2).$$
We can note that for fixed $r$, this quantity  converges as $\delta \to 0$ to 
$$  V(r):= 2 \int\int dx dy f(x) f(y) \left( G^{r,r} (x, y)^2 +G (x, y)^2 - 2 G^r(x,y)^2\right),$$ 
where $G$ is equal to the continuous Green's function and $G^r, G^{r,r}$ the corresponding meaned out versions of $G$ (this is because the three functions $G_\delta, G_\delta^r, G_\delta^{r,r}$ converge respectively to $G, G^r, G^{r,r}$ uniformly on the region where $d(x,y)\ge \eps$ for any given $\eps>0$, and the integral on the region where $d(x,y)<\eps$ can be shown to be bounded by an $o(\eps)$ uniformly over $\delta$ -- we leave the details to the reader).
It is also easy to check that $V(r)$ converges to $0$ as $r$ goes to $0$ (this is because $G^{r,r}, G^r$ both converge uniformly to $G$ for $(x,y)$ such that $d(x,y)\ge \eps$ and the 
integral on the region where $d(x,y)<\eps$ can be controled uniformly in $r$).

For fixed $r$, we know that when one lets $k \to \infty$, the random variable $Y (r, \delta_k, f)$  converges in distribution to 
$$ Y(r,f) := \int \frac{1}{\pi r^2}\left( \phi (B_r(x))^2-E[ \phi (B_r(x))^2]\right) f(x) dx -  [ T (f) ].$$
Hence by Fatou's lemma, we have $E[ Y(r,f)^2 ] \le V(r).$ 
But as $r \to 0$, $V(r)$ converges to $0$  
which implies that  $Y(r,f)$ converges to $0$ in $L^2$. 
On the other hand, when $r \to 0$, the definition of $\phii$ shows that  $Y(r,f)$ converges in $L^2$
 to $\phii (f) - T (f)$, so that we can  conclude that $T(f) = \phii (f)$ almost surely. 

Finally, in order to show that this coupling of $( {\mathcal L}, T, \phi )$ does fulfil all the conditions of the lemma, it 
only remains to check that for all given $A$, conditionally on $\tilde A$, the conditional distribution of the restriction of $\phi$ to $\tilde A$ is just a GFF in $\tilde A$. This follows immediately from the fact that the corresponding result holds on the cable systems.  
\end {proof}

Note that this lemma shows in particular that the complement of $\tilde A$ is a local set for the GFF $\phi$ (in the definition introduced by Schramm and Sheffield \cite {SchSh2}, see also \cite {Wgff} for a survey). We now explain how to deduce the proposition from the lemma: 

\begin {proof}[Proof of Proposition \ref {p1}]
Let us consider the coupling given by Lemma \ref {lt} of a GFF $\phi$ with a loop-soup $\Lambda$ (so that the square of the GFF corresponds to the occupation time of the GFF), and let us define on the one hand a CLE$_4$ as outer boundaries of the loop-soup clusters CLE$_4$ and on the other hand the ``level-line CLE$_4$'' defined from $\phi$.

Because of conformal invariance, and because there are only countably many loops, in order to prove that these two CLE$_4$'s are in fact identical, it suffices to prove that the CLE$_4$ loop that surrounds the origin is almost surely the same for both (by conformal invariance, it implies that the loop that surrounds a given point $z$ coincide as well). 
Let $\gamma_0$ and $\tilde \gamma_0$ denote these two loops (the former is the outer boundary of the loop-soup cluster around the origin, and the latter is the outermost GFF level line surrounding the origin). We know that they have the same distribution, so that their conformal radii (i.e. the conformal radii of their interiors, as seen from the origin) have the same distribution. In order to prove that $\gamma_0$ and $\tilde \gamma_0$ are almost surely equal, it therefore suffices
to show that almost surely, $\tilde \gamma_0$ lies in the closure of the interior of $\gamma_0$, i.e. that no point with rational coordinates inside of $\tilde \gamma_0$ lies to the outside of the loop $\gamma_0$.

\begin{figure}[ht!]
\begin{center}
\includegraphics[scale=.55]{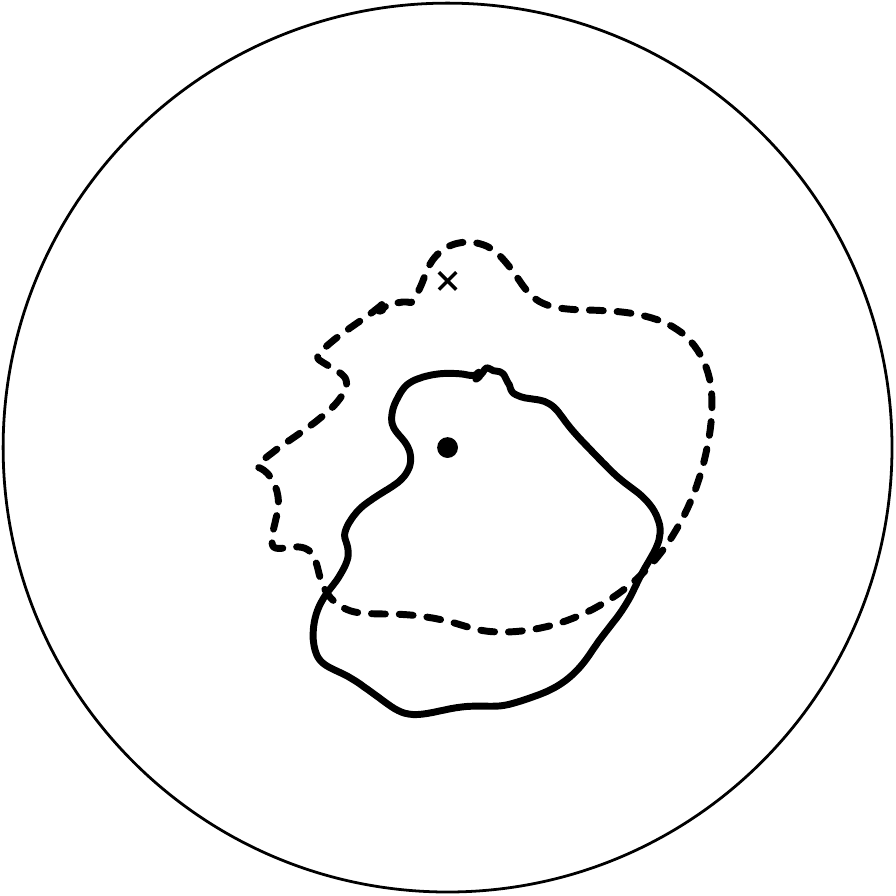}
\includegraphics[scale=.55]{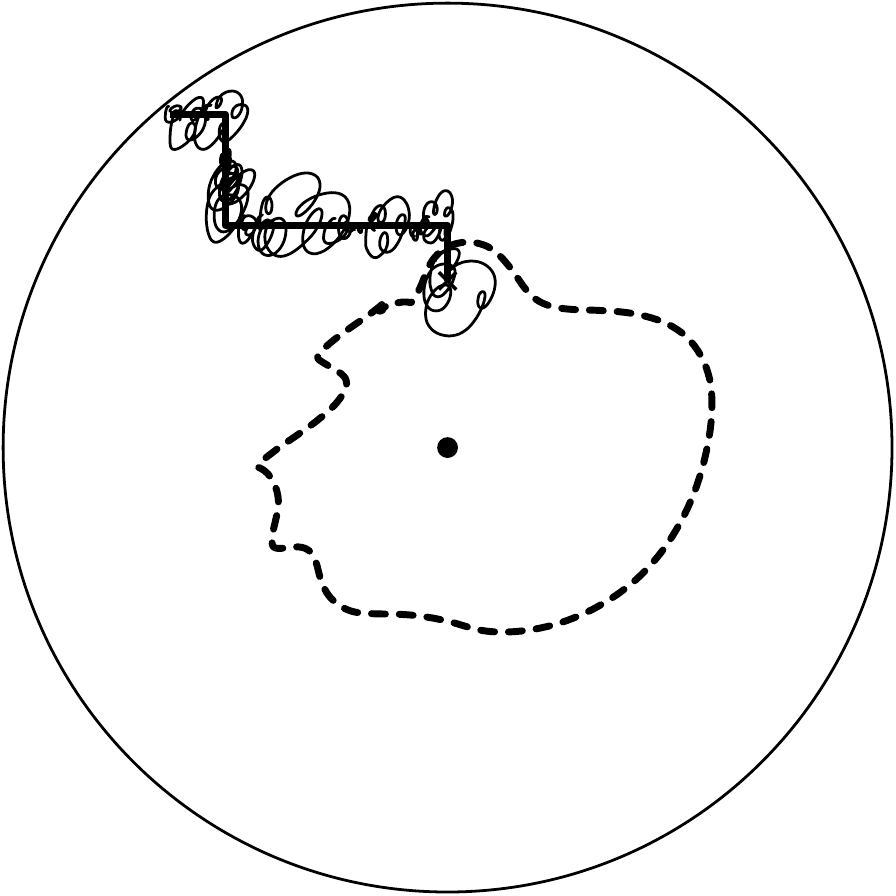}
\includegraphics[scale=.55]{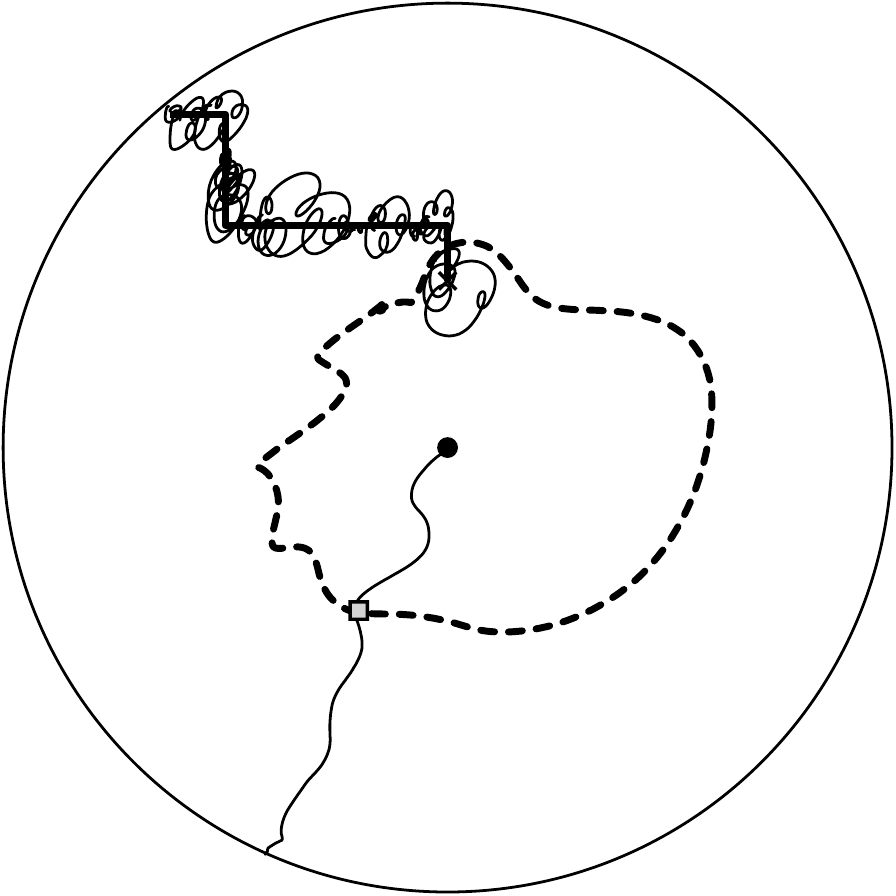}
\caption {(a) $\gamma_0$ (plain), $\tilde \gamma_0$ (dashed), and a point $z_0$ outside of $\gamma_0$ but inside of $\tilde \gamma_0$. (b) A well-chosen path from $\partial \U$ to $z_0$ with the cluster it intersects. (c) The loop $\tilde \gamma_0$ intersects $\tilde A$ and goes out of $\tilde A$. \label {figure5}}
\end{center}
\end {figure}

Suppose that some point with rational coordinates inside of $\tilde \gamma_0$ lies to the outside of $\gamma_0$. Because of the fact that CLE$_4$ is locally finite, one could then actually find a continuous path $\eta$ made of finitely many horizontal or vertical segments of lines with rational $x$ or $y$-coordinates that joins this $z_0$ to the unit circle, and such that when one attaches to $\eta$ all interiors of loop-soup clusters to it intersects, then one has still not yet discovered $\gamma_0$, nor disconnected $\gamma_0$ from $\partial \U$. If we would have set $A$ to be that $\eta$, then it means that the connected component of $\tilde A$ that contains the origin also has part of $\partial \U$ on its boundary (because $\tilde A$ does not disconnect $\gamma_0$ from the unit circle). If we then draw a continuous path in that open set from the origin to the unit circle, then it necessarily intersects 
$\tilde \gamma_0$, because this loop surrounds the origin. So, there is some point of $\tilde \gamma_0$ that lies in the connected component of $\tilde A$ that contains the origin. 
On the other hand, the loop $\tilde \gamma_0$ cannot entirely lie in $\tilde A$ because it surrounds $z_0$ and therefore intersects $\eta$. Hence, $\tilde \gamma_0$ contains at least one ``excursion'' away from the boundary of $\tilde A$ to its inside (see  Figure \ref {figure5}). 

Let us now fix a point with rational coordinates and a path $\eta$ made of finitely many horizontal and vertical segments with rational $x$ respectively $y$ coordinates. We set $A = \eta$ and define $\tilde A$ as above. 
Recall that $\tilde \gamma_0$ is a level line of the GFF $\phi$ with $0/\pm 2 \lambda$ values on its two sides, and that it is a deterministic function of the GFF. Conditionally on $\phi$, it is  therefore independent from everything, and in particular from the local set $\tilde A$. Recall also that conditionally on $\tilde A$, $\phi$ restricted to $\tilde A$ is a GFF in that set with zero boundary conditions, from which one can easily deduce (using similar arguments as in \cite {ASW}) that almost surely, $\tilde \gamma_0$ does either entirely stay on $\tilde A$ or does not intersect $\tilde A$ at all. 

Hence (because there are countably many such possible points with rational coordinates and paths $\eta$), we conclude that almost surely, 
every point that is surrounded by $\tilde \gamma_0$ is also surrounded by $\gamma_0$, which concludes the proof.
\end {proof}

\section {Conclusion of the proof of Theorem \ref {mainthm}} 
\label {S4}

In this section, we will complete the proof of Theorem \ref {mainthm}. 
What remains to be shown is that if one considers a loop-soup with $c=1$ in $\U$ and conditions on the outermost loop-soup cluster boundary $\gamma$ that surrounds the origin, 
then the union of the loops that touch $\gamma$ will form the union of a  Poisson point process of excursions with intensity $\beta=1/4$. We will do this in two steps; first we will show that this is 
true for some value $\beta$, and then we will show that in fact $\beta=1/4$.

Let us first recall a few features related to Brownian excursions and to the square of the GFF: 
Consider on the one hand a Poisson point process ${\cal E}$ of Brownian excursions in the unit disc in $\U$ with intensity $\beta >0$ (we will use the normalization of the excursion measure, so that the union of the excursions away from the upper half-circle define a one-sided restriction measure with exponent $\beta$, see \cite {Wcrrq}). This Poisson point process defines an occupation time field $\T_{\beta}$ in the unit disc (note that for any $\eps >0$, there are almost surely only finitely many excursions of ${\cal E}$ of diameter greater than $\eps$, so that for each
given domain $D$ that is at positive distance from $\partial \U$, the occupation time $\T_\beta (D)$ is almost surely finite -- another simple way to see this is to note
that the expectation of the field $\T_\beta$ is a multiple of the Lebesgue measure in $\U$). It is easy to see that the occupation field is 
determined by the 
trace of the union of all the excursions (for instance, $\T_\beta (D)$ is the appropriately scaled limit when $\eps \to 0$ of the area of the $\eps$-neighbourhood of the union of all the excursions in $D$, see \cite {LG}), and conversely, it is clear that the trace of the union of all excursions is the support of $\T_\beta$. Note also that this field is in fact a subordinator with respect to $\beta$ (because if $\T_\beta$ and $\T_{\beta'}$ are chosen to be independent, then $\T_{\beta} + \T_{\beta'}$ is distributed as $\T_{\beta + \beta'}$). 
We also define the centered occupation time field $\hat \T_\beta = \T_\beta - E(\T_\beta)$ (again, the expectation is finite, so that there is no definition difficulty here --  as noted above,  $E( \T_\beta )$ is a constant multiple of $\beta$ times the Lebesgue measure). 
We now consider also the square of a Gaussian free field $\phii$ in $\U$ that is independent from these excursions, and we will be interested in the field $\frac12\phii +  \hat\T_\beta$. 

On the other hand, for all real $u$, one can define a new field $\phiiu$, which is the ``recentered square'' of $\phi + u$. 
This can be done in several equivalent ways. 
One possibility is to note that the field $\phi + u $ is absolutely continuous with respect to $\phi$, when restricted to a set at positive distance of the unit circle. One can then consider $\phiiu$ to be the field obtained by recentering (i.e. subtracting its expectation) the field obtained by taking the corresponding Radon-Nikodym derivative of $\phii$. It is easy to check that this field is equal 
to $\phii + 2 u \phi $.

Then, one has the following well-known identity in distribution, often referred to as (a version of) Dynkin's isomorphism theorem (see for instance Sznitman \cite {Sz} for such a statement {in the discrete case and Sznitman \cite{Sz2} for the statement in the continuous case} -- it can be also viewed and understood as a consequence of the spatial Markov properties of $c=1$ loop-soups in the spirit of the recent results in \cite {Wresampling,CL}): 
\begin {proposition}[``Dynkin's isomorphism'']
\label {p2}
For some constant $k$ independent of $u$, the two fields  $\frac12(\phiiu+u^2)$ and $\frac12\phii + \T_{ku^2}$ have the same distribution.
In particular, if we subtract the  means of both sides, $\frac12\phiiu$ and $\frac12\phii+\hat\T_{ku^2}$ have the same distribution.
\end {proposition}
We will come back to the issue of what the value of $k$ actually is (it is in fact $1/ (2 \pi)$ with our normalization choices) in a few paragraphs, but let  us first combine this proposition with the previous couplings between the GFF and CLE$_4$:

Suppose that $\phii$ and the loop-soup (that defines $\gamma$) are coupled as in Proposition \ref {p1}. 
On the one hand, we know that conditionally on  $\gamma$, 
the field $\phi$ restricted to $O(\gamma)$ is distributed like $\pm 2 \lambda$ plus a GFF in $O (\gamma)$. Hence, conditionally on $\gamma$, the distribution of its recentered square is exactly the conformal image (via the conformal map $\psi^{-1}$ from $\U$ into $O(\gamma)$) of the law of $\phiiu$ for $u=2 \lambda$. But Proposition \ref {p2} now shows that this is exactly the distribution of the conformal image via $\psi^{-1}$ of 
$\phii + \hat \T_{\beta}$  where $\beta=ku^2$ (recall that we have chosen $\phii$ and $\hat \T_{\beta}$ to be independent). 

On the other hand, our previous decomposition of the loop-soup inside $O(\gamma)$, and the fact that the loop-soup occupation times define the square of a GFF, imply 
that conditionally on the loop $\gamma$, the conditional distribution of the square of the GFF inside $O (\gamma)$ minus its conditional expectation given $\gamma$, is the sum of the square of a GFF in $O(\gamma)$ with the 
(recentered) occupation time of the union of all loops of $\Lambda_0^b$ (and these two fields are also independent, conditionally on $\gamma$). In other words, this is the distribution of the image under $\psi^{-1}$ of the sum of the centered occupation time measure $\tilde \T^b$ of $\tilde \Lambda^b$ with an independent squared Gaussian free field in $\U$. 

Hence, we get that conditionally on $\gamma$, the sum of  $\tilde \T^b$ with an independent squared GFF in $\U$ is distributed like the sum of $\hat \T_\beta$ with an independent squared GFF in $\U$. 
Recall also that $\tilde \T^b$ is independent of $\gamma$, and that $\hat \T_\beta$ is also independent of $\gamma$, so that this is in fact an identity in distribution, unconditionally on $\gamma$.
It follows that for each smooth test function $f$ with support that is at positive distance of the unit circle and for any $t$ for which
$E( \exp (it\phii (f))) \not= 0$ (recall that this characteristic function is well-understood, see for instance its expression in terms of the Brownian loop-measure
that we recalled in item (c) of Section \ref{S3}, so that it is easy to check that this set of $t$'s is dense in the real line), we have
$$E \bigl( \exp (it\hat \T_\beta (f)) \bigr) = E \bigl( \exp ( it \tilde \T^b (f)) \bigr) .$$
It therefore follows that $\hat \T_\beta (f)$ and $\tilde \T^b (f)$ are identically distributed. 

But since $f \mapsto (\tilde \T^b (f), \hat \T_\beta (f))$ is linear, this identity implies that the characteristic function of any finite marginals of $\hat \T_\beta$ and $\tilde \T^b$ are identical i.e. that 
 the random fields $\hat \T_\beta$ and $\tilde \T^b$ have the same distribution. 
This proves the final statement in our theorem, except that we have not yet determined the value of $\beta$.

\medbreak

In order to show that  $\beta = 1/4$, we just need to do a bookkeeping of the constants involved in the previous argument (note that in the next section, we will also describe a heuristic argument that explains why $\beta$ has to be indeed $1/4$, that can also be turned into a -- somewhat convoluted -- proof): 

Let us first consider the Brownian excursion measure $M$ in the upper half-plane defined as the limit when $\eps$ goes to $0$ of $( \pi / \eps ) $ times the integral over $x \in \R$ of the law of 
Brownian motion started from $x + i \eps$ and stopped upon exit of the upper half-plane. Let us first see how to work out for the restriction exponent $\alpha$ the Poisson point process of Brownian excursions with intensity $M$ restricted to the excursions that start and end on the negative half-line:
It is easy to see that the  $M$-mass of the set of excursions that start and end on $[-2, -1]$ and that intersect the imaginary half-line is equal to $\log (9/8)$ because it is also equal (using  a reflection argument) to the $M$-mass of the set of excursions that start on $[-2, -1]$ and end on $[1, 2]$. The probability that a Poisson point process of excursions with intensity $M$ does not intersect the 
imaginary half-line is therefore $8/9$. It therefore follows easily that $\alpha = 1$ (by the definition of the restriction formula, and the explicit square map from the top-left quadrant onto $\HH$).

On the other hand, it is easy to see that integral over $M$ of the occupation time density at any point $z \in \HH$ is equal to $\pi / 4$. Indeed, it is equal to 
$$ \lim_{\eps \to 0} \frac { \pi }{ \eps} \int_{\R} G_\HH (x+ i \eps, z) dx =   \lim_{\eps \to 0} \frac { \pi }{  \eps} \int_{\R} G_\HH (z, x+ i \eps) dx $$ 
which is easily shown to be equal to $\pi $ (because the expected local time at height $\eps$ of a one-dimensional Brownian motion started from $\Im (z) > \eps $ and stopped at its first hitting time of $0$ is equal to $\eps$).
This implies that the expected density of $\T_{ku^2}$ is constant and equal to $k u^2 \pi$  (i.e. the expected value of the total cumulated time spent in an open set $O$ by all the excursions of the Poisson point process is $k u^2 \pi$ times the area of $O$).
Comparing this with the first  identity in law in Proposition \ref{p2}, we get that $k= 1/ (2\pi)$.  
Similarly, the value of $\beta$ corresponds to $\beta\pi=(2\lambda)^2 / 2=\pi/4$, so that we can conclude that $\beta=1/ 4$.

\section {Remarks}
\label {S5}

\subsection {A heuristic justification for the value of $\beta$} 
We now outline an argument that explains why $\beta = 1/4$ using the relation between restriction measures, loop-soups and SLE derived in \cite {WW2}.
This will allow us to make some further comments on the structure of the clusters and make the link with some other features (note for instance that this argument will not use the relation to the GFF).

Let us first recall the following result from \cite {WW2}: Consider a Poisson point process of Brownian excursions in the unit disc with intensity $\alpha$, but restricted only to those excursions that have both their end-points on the upper semi-circle. 
Then (see \cite {Wcrrq}), the lower boundary of the union of all these excursions is a simple curve $\eta (\alpha)$ from $-1$ to $1$ in $\overline \U$ that can be described in terms of restriction measures, or alternatively as a SLE$_{8/3} ( \rho)$ process for $\rho = \rho (\alpha)$. If one adds to this picture an independent loop-soup in $\U$ with intensity $c \le 1$, one can now look at the union of $\eta (\alpha)$ with all the loop-soup clusters that it intersects, and consider its lower boundary $\eta (\alpha, c)$. Then, as shown in 
\cite {WW2}, this is a simple curve from $-1$ to $1$ in $\overline \U$, that is distributed like an SLE$_\kappa (\rho)$ process, for $\kappa = \kappa (c)$ and some explicit 
$\rho$ depending on $\alpha$ and $c$ (this fact is actually instrumental in the derivation of Lupu's result \cite {Titus} that we used in the previous section).

Standard computations involving Bessel processes allow to describe simple features about these SLE$_\kappa (\rho)$ processes.   
This implies for instance that for $c=0$, $\eta (\alpha)$ touches the upper half-circle if and only if $\alpha< 1/3$ (see \cite {LSWr,Wcrrq}). 
Similarly, the value of $\alpha$ for which $\eta (\alpha, 1)$ is exactly an SLE$_4$ is $\alpha = 1/4$. When $\alpha < 1/4$, the path $\eta (\alpha ,1 )$ does touch the upper half-circle, while 
when $\alpha > 1/4$, the probability that it gets $\eps$-close to some given subarc of the upper half-circle is bounded by some power (that depends on $\alpha$) of $\eps$ as $\eps \to 0$.   

We can note that these features of the paths $\eta (\alpha, c)$  also lead to similar properties for clusters obtained by considering the superposition of a Poisson point process of excursions with intensity $\alpha$ in the unit disc with intensity $\alpha$ and no restriction on the end-points, with  a loop-soup with intensity $c$. If one looks at the union of all excursions with the loop-soup clusters that they intersect, then the boundary of the connected component of the complement of this set that contains the origin will
intersect the unit circle with positive probability if and only if the previous $\eta (\alpha, c)$  touches almost surely the upper half-circle.  

Let us also recall that an SLE$_4$ is a simple curve, but that it gets rather close to having double-points (recall that SLE$_\kappa$ curves for $\kappa > 4$ do have double points). 
As opposed to SLE$_\kappa$ for $\kappa < 4$ where this probability decays in a power-law fashion, the probability that an SLE$_4$ curve from $-1$ to $1$ in the unit disc does behave as depicted 
in the left-hand part of Figure \ref {downcrossings} decays slower than any power-law of $\eps$ as $\eps \to 0$ (it decays logarithmically). Using the description of CLE$_4$ loops by SLE$_4$, it follows that 
with a probability that is (asymptotically) larger than any power of $\eps$, one finds a loop of the kind depicted on the middle picture of Figure \ref {downcrossings}. 
But, by resampling some set of macroscopic Brownian loops in the loop-soup, we conclude that with probability (asymptotically) larger than any power of $\eps$, the inner boundary of the outermost cluster surrounding the origin does come $\eps$-close to its outer boundary (see the right picture in Figure \ref {downcrossings}). It finally follows that the same is true for the inner boundary of $\Lambda_0^b$ as well (as it is in-between the inner and outer boundaries of the cluster).

\begin{figure}[ht]
\begin{center}
\includegraphics[width=2in]{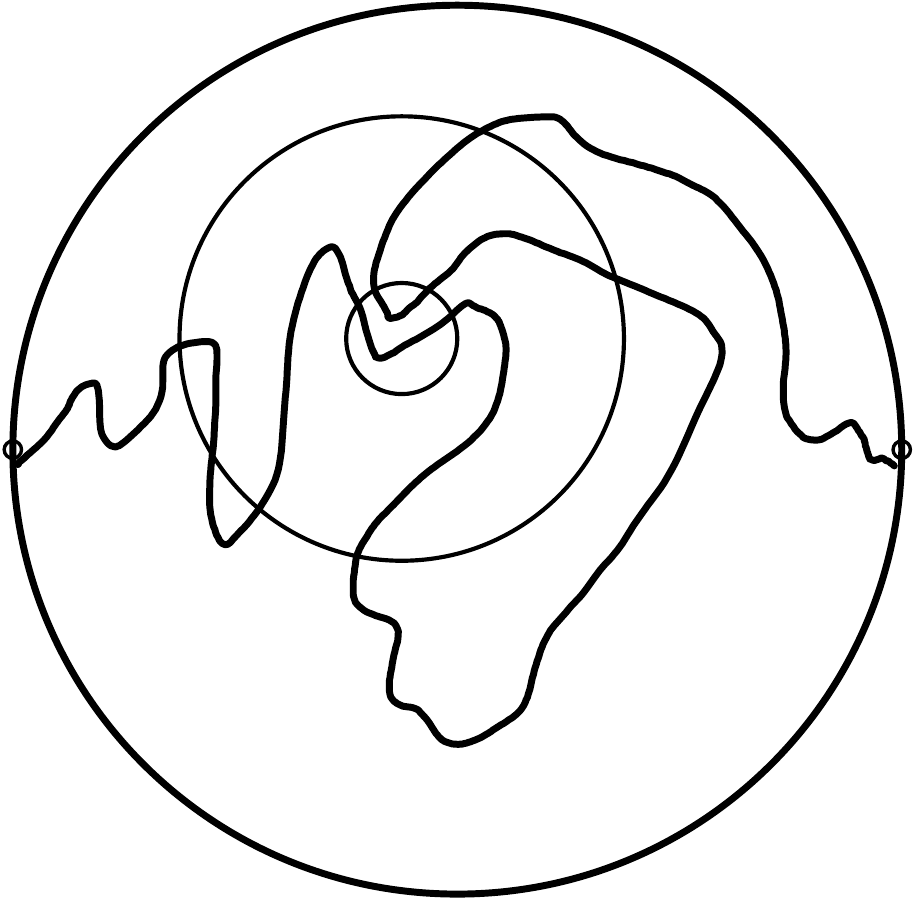}
\includegraphics[width=2in]{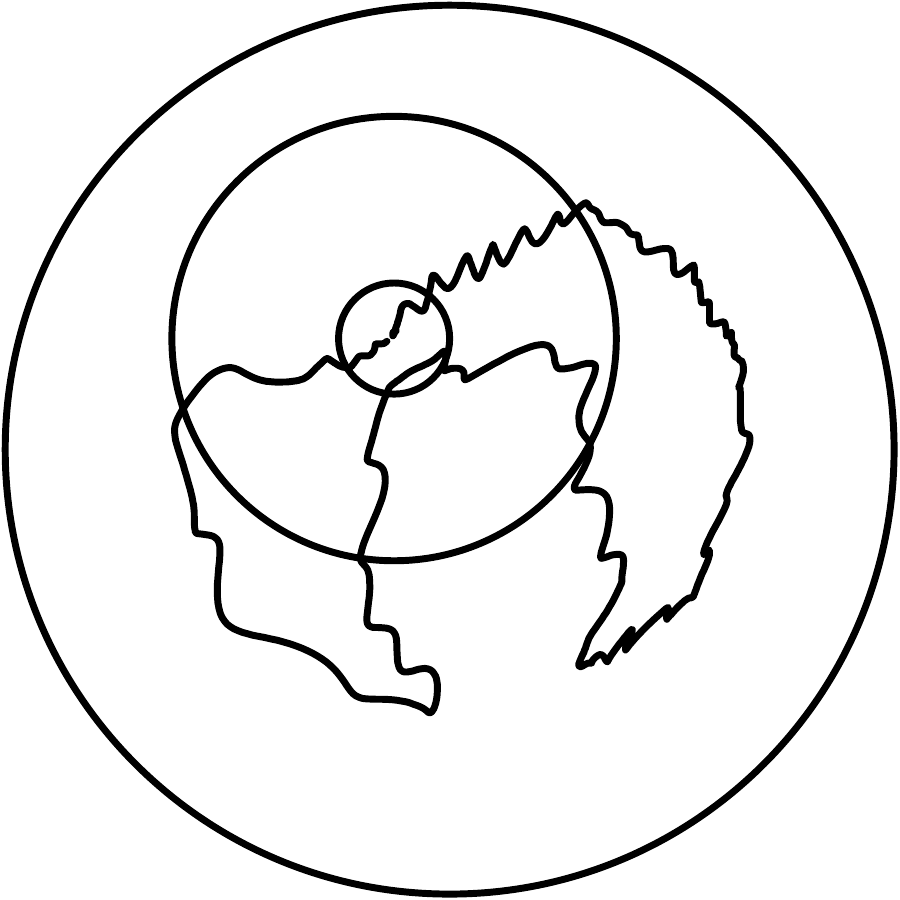}
\includegraphics[width=2in]{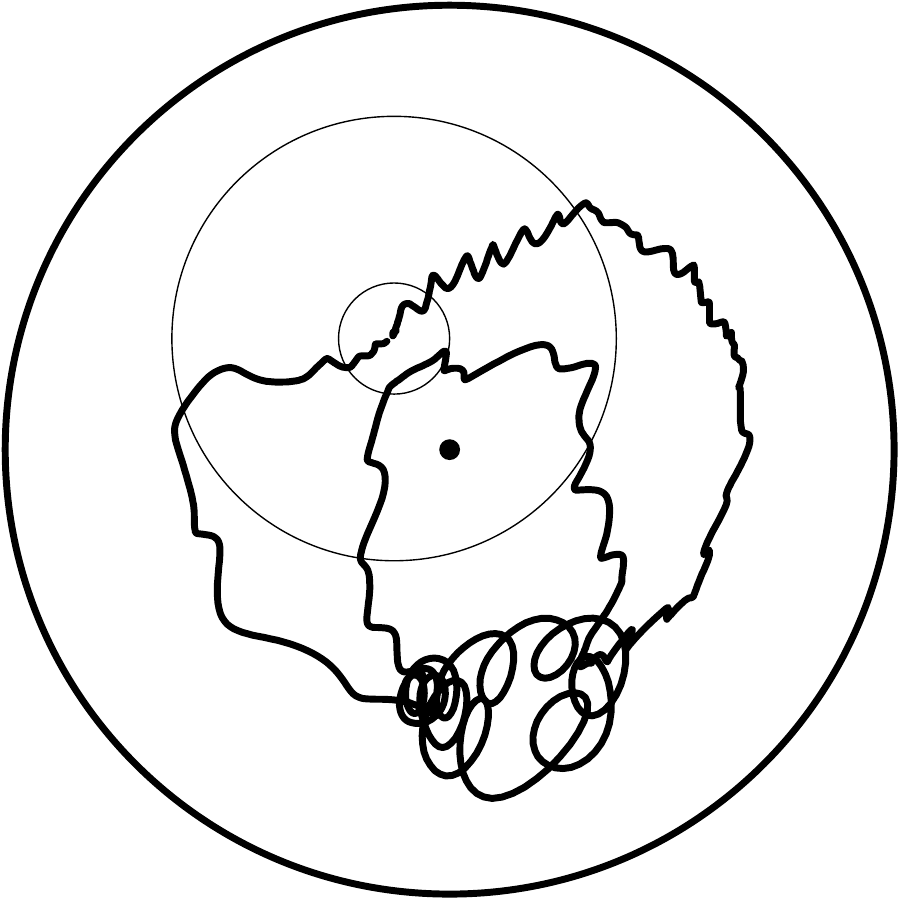}
\end{center}
\caption{\label{downcrossings} Double crossing of some (non-prescribed) annuli of radii $\eps$ and $1/2$ by an SLE$_4$ or a CLE$_4$ loop and the inside and outside boundaries of a cluster.}
\end{figure}

But, if we compare our construction of the cluster surrounding the origin (conditionally on $\gamma$) with the previously recalled properties of $\eta (\alpha, 1)$, we can deduce that $\beta \le 1/4$ (otherwise, the probability that the inside boundary would come close to the outer boundary would decay in a power-law fashion). 

Let us now explain why it is on the other hand not possible that $\beta < 1/4$. In that case, then the properties of SLE$_4 ( \rho)$ for $\rho < 0$ (i.e. that this process touches the boundary of the domain) shows that with positive probability, there exists a fractal set of local cut points for the outermost cluster 
surrounding the origin, on its outer boundary (see Figure \ref {cutpoints}). But by elementary topological considerations (and using the fact that a Brownian loop has no cut points i.e. that a Brownian motion has no double cut points), it follows that all these cut points belong to the same single Brownian loop that we call $l$, and to no other loop in the loop-soup. 
\begin{figure}[ht!]
\begin{center}
\includegraphics[scale=0.55]{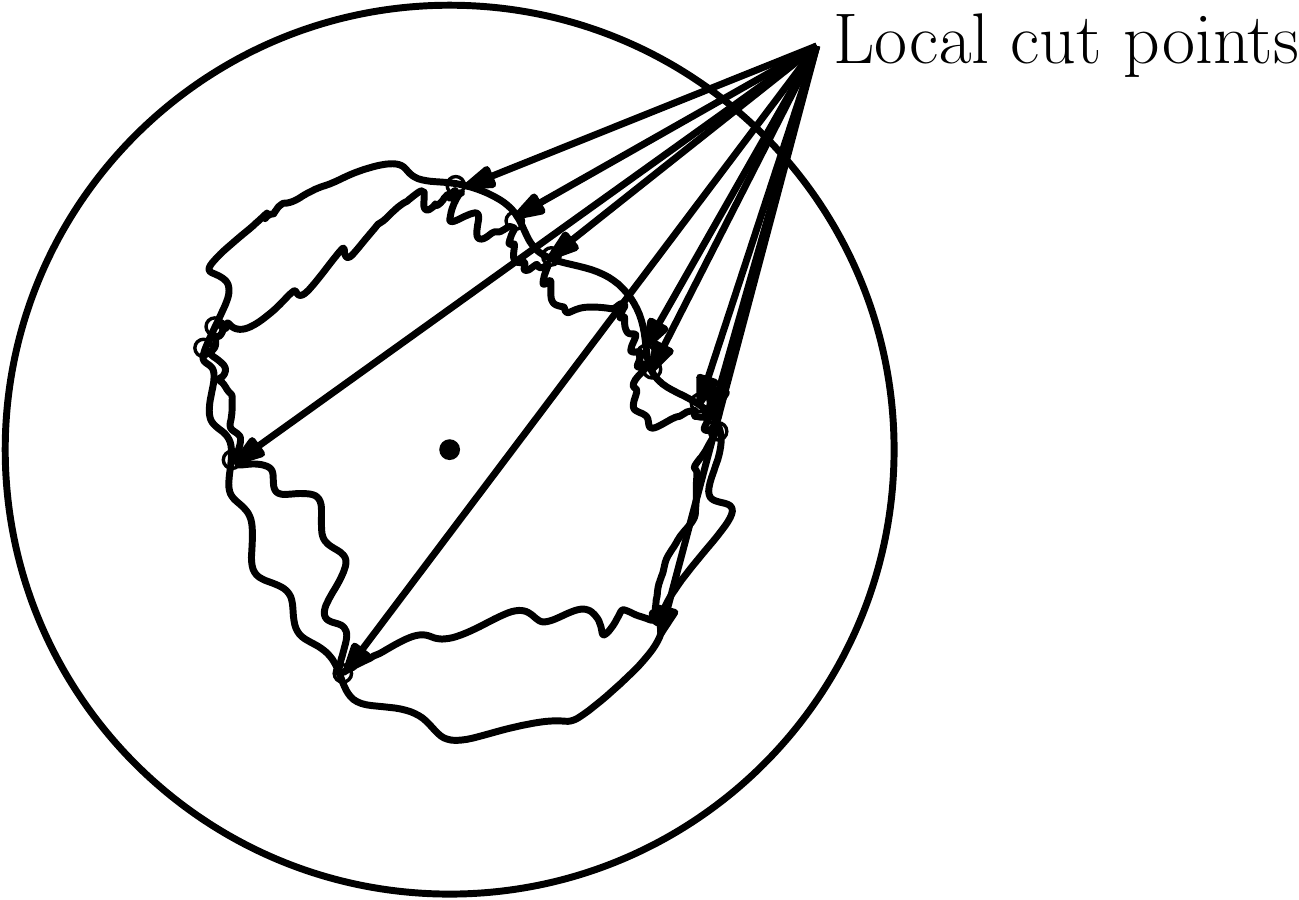}
\includegraphics[scale=0.55]{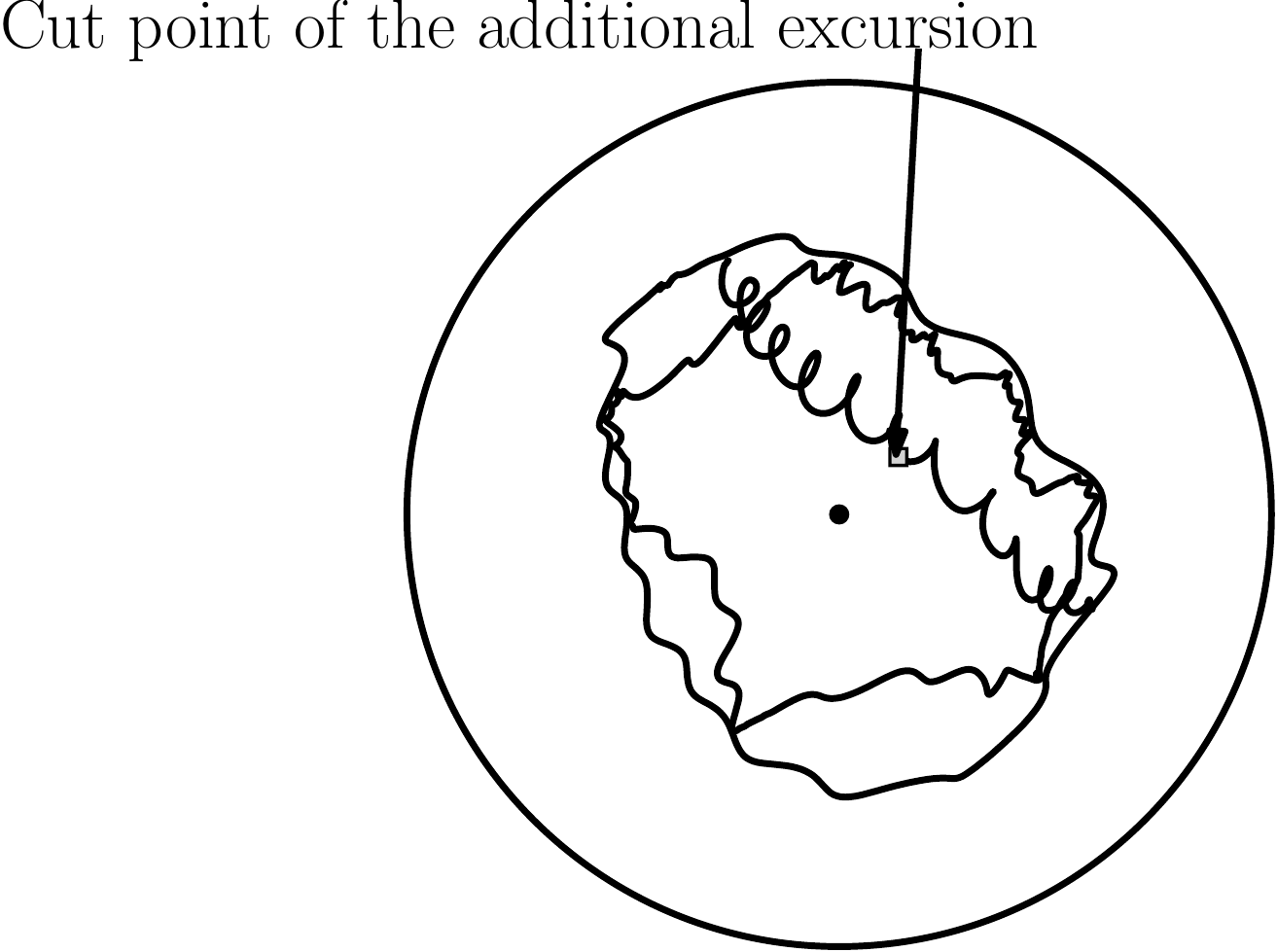}
\end{center}
\caption{\label{cutpoints} (a) The inner and outer boundary of the cluster and the cut points. (b) Adding an excursion that causes a topological contradiction}
\end{figure}
But if one adds to this picture a single excursion (and this can happen, when one resamples a subset of the Poisson point process of excursions) with a cut-point in its the middle, as in Figure \ref {cutpoints}, then this additional excursion would have to belong to a Brownian loop, and therefore also passes through some of the same cut points as $l$ was passing through. This leads to a contradiction because the loop that goes around the whole set of cut points cannot exist any more, and the cut points on the other side are on no loop, so that the cluster is not a cluster any more. 
Hence, this shows that the inside boundary of a $c=1$ loop-soup cluster does not touch its outer boundary, and henceforth that $\beta \ge 1/4$. 

\subsection {Further comments on cut points}
As we have already mentioned,  when adding a restriction measure with exponent $1/4$ to a loop-soup with intensity $c=1$, one can reconstruct exactly a SLE$_4$ (see \cite {WW2}). This result is still valid for other values of $c$ (choosing $\kappa = \kappa (c)$ and $\beta(c) = ( 6 - \kappa) / 2 \kappa$, see \cite {WW2}), which raises naturally the question whether the last statement of the theorem could actually be generalized to other $c$'s as well (with an appropriate choice of intensity for the Poisson point process of excursions, depending on $c$) 
and whether it is only our method of proof via the GFF that 
does not extend to the general case $c<1$. Let us now informally explain why we believe that this decomposition with a Poisson point process of excursions
is in fact specific to $c=1$.

\begin{figure}[ht!]
\begin{center}
\includegraphics[scale=0.5]{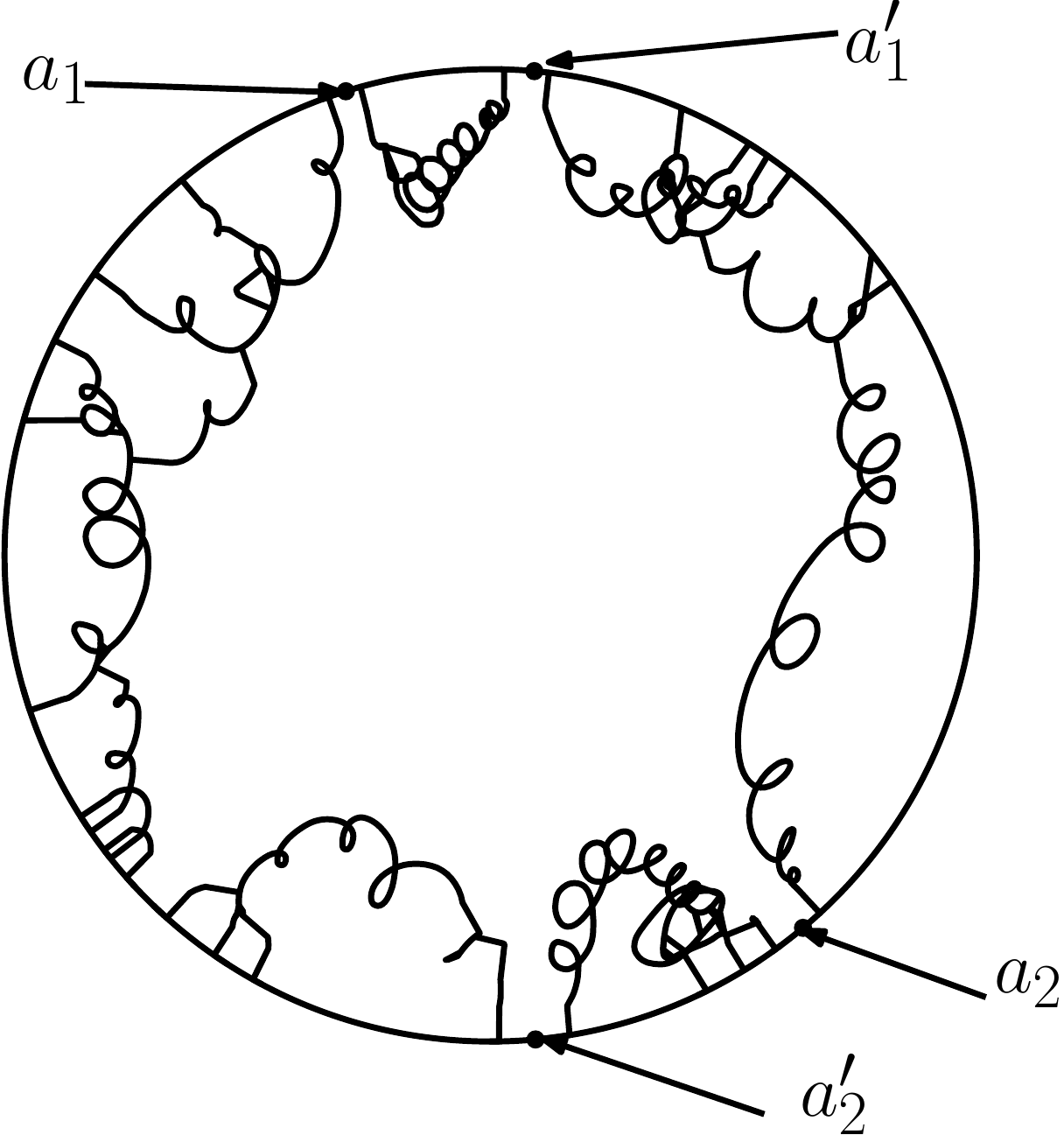}
\includegraphics[scale=0.5]{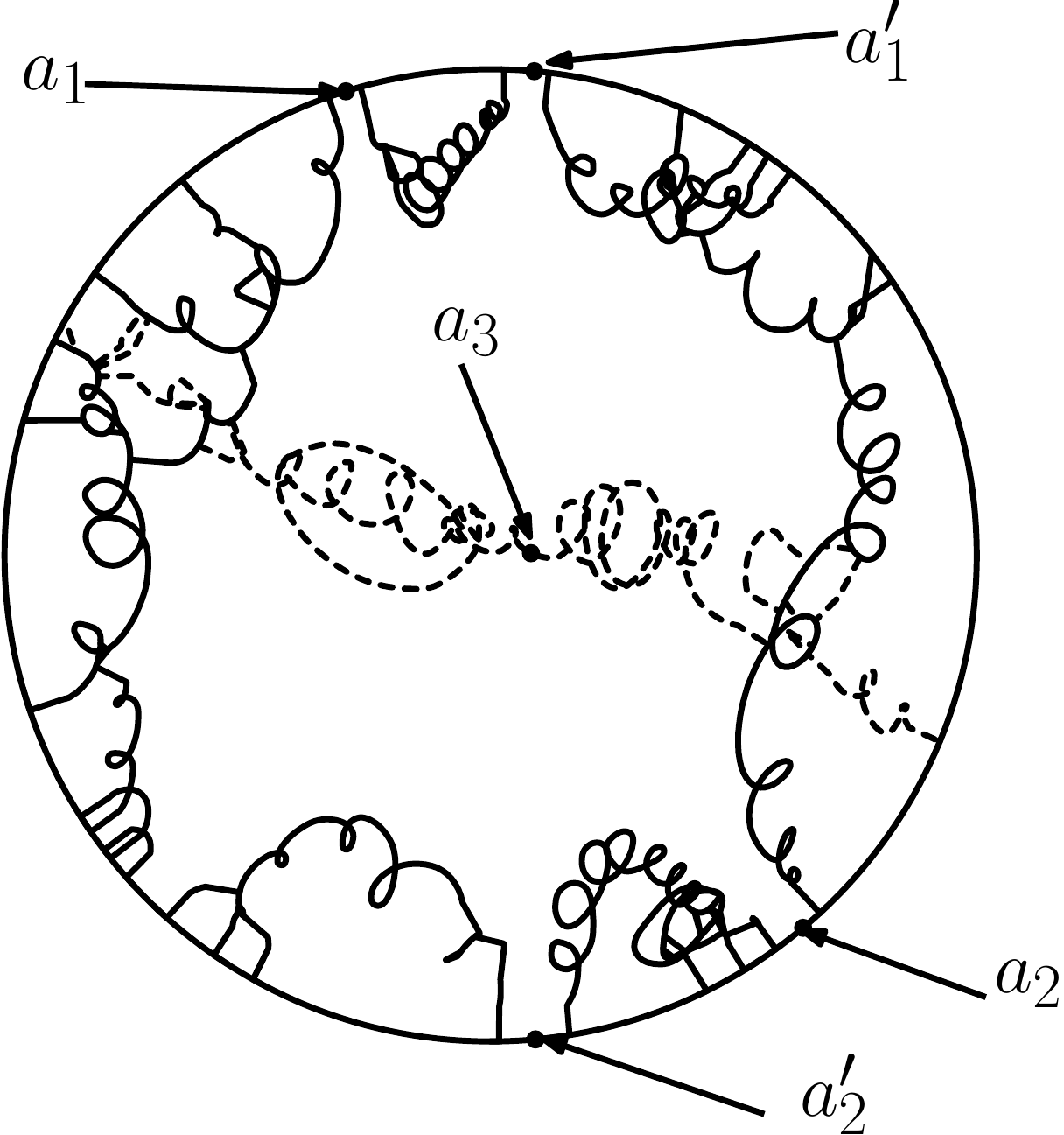}
\end{center}
\caption{\label{fig1}Two possible configuration with cut points.}
\end{figure}

But before this, we will make some comments on the $c=1$ case, in the same spirit as in our identification of the value of $\beta$, but focusing on the 
Poisson point process of excursions only (and not on the entire cluster and its boundary). 
Let us first note that when $\beta=1/4$, the inner boundary (i.e. the boundary of the connected component of its complement that contains the origin) of the union of all the Brownian excursions does touch the unit circle with positive probability. This comes from the fact that a Poisson point process of excursions in $\U$ restricted to those excursions that start and end on a half-circle, defines a one-sided restriction measure with exponent $1/4$ which is smaller than $1/3$ (which is known to be the critical value for its boundary touching this half-circle, see \cite {LSWr,Wcrrq}).  In other words,  the union of 
the excursions is not necessarily connected. This implies in particular that the loops of $\Lambda_0^b$ alone do indeed not form a connected cluster. It is only when adding the loops of $\Lambda_0^i$ that one (almost surely, when one conditions on $\gamma$ and $\Lambda_0^b)$ connects them into a cluster (but we have seen that this happens with probability one). 
This can be easily seen as follows: With positive probability, the configuration of $\Lambda_0^b$ creates four ``cut points'' $a_1$, $a_1'$, $a_2$ and $a_2'$ seen from the origin that separate $-1$ from $1$ on the 
upper half-circle and on the lower half-circle respectively as in Figure \ref {fig1} (this is due to the possible existence of cut points on the aforementioned one-sided restriction measure samples). Note that if the loops of $\Lambda_0^b$ alone would form a connected cluster, then (because Brownian loops can not pass twice through local cut points, see \cite {BL}), the points $a_1$, $a_1'$, $a_2$ and $a_2'$ are all visited by the same Brownian loop (and in cyclic order). 

But, it is also possible (with positive probability) to have also on top of the previous picture,  an additional excursion (and just one), with a cut point $a_3$ in its middle as represented in dashed in Figure \ref {fig1}, that joins the neighbourhood of $1$ to the neighbourhood of $-1$. 
In such a case (again because we know that a Brownian loop has no double points that are also cut points, see \cite {BL}), it follows from elementary topology that the cut points $a_1$ and $a_2$ can not simultaneously be local cut points of the same loop. In fact, exactly one of the two points will be a local cut point of a loop but not the other one (otherwise, there is no way in which the excursion that $a_3$ is part of can be closed into a loop). This implies readily that with positive probability, the loops of $\Lambda_0^b$ alone do not create a unique cluster of loops (because either the part between $a_1$ and $a_1'$, or the part between $a_2$ and $a_2'$ is disconnected from the other part). A simple $0-1$ law argument then implies that this is almost surely the case. Similarly, we see that some local cut points of Brownian loops remain local cut-points of the closure of $\tilde \Lambda^b$. 
This all indicates that the way in which one tries to reconstruct the loops out of the union of the excursions is a rather tricky and non-local procedure. The fact whether a local cut point of $K$ is on a loop will be correlated with the existence of some other excursions and cut-points far away.  

This type of argument shows that the Poissonian decomposition of the inside part of cluster cannot hold in the $c \to 0$ limit. Indeed, the interior of a Brownian loop when conditioned on its outer boundary is {\em not} distributed like a Poisson point process of Brownian excursions with some intensity. If this would have been the case, such a process would also have created local cut-points (because the loop has local cut-points), and the same topological construction as in Figure \ref {fig1}  would not be possible to topologically correspond to a single Brownian excursion (because a Brownian motion almost surely has no double cut points). 

\subsection {Relation to Markovian resampling of part the loop-soup}
The apparent wonder that leads to decomposition into excursions can be related and enlightened by the resampling property of the $c=1$ loop-soup pointed out in \cite {Wresampling,CL}. Let us give an informal description of this: Suppose that one ``discovers'' the loop $\gamma$ from its outside, and then explores in 
both directions all pieces of loops that touch $\gamma$, up to the points where  their image under a given conformal transformation from the 
interior of $\gamma$ into the unit disc reaches distance $\epsilon$ from the unit circle. For each $\epsilon$, only finitely many such ``beads'' do reach distance $\eps$, and then, the way to complete them is described in \cite {Wresampling} and it is basically a Poisson point process of bridges with given set of endpoints. When $\eps$ tends to zero, the bridges become excursions and the set of endpoints becomes closer to a multiple of the Lebesgue measure on the circle (when appropriately renormalized), so that  it is not that surprising to obtain a Poisson point process of excursions in the limit (though making this argument rigorous seems non-trivial).  Recall that this resampling property is very specific to this $c=1$ case.

\begin{figure}[ht!]
\begin{center}
\includegraphics[scale=0.60]{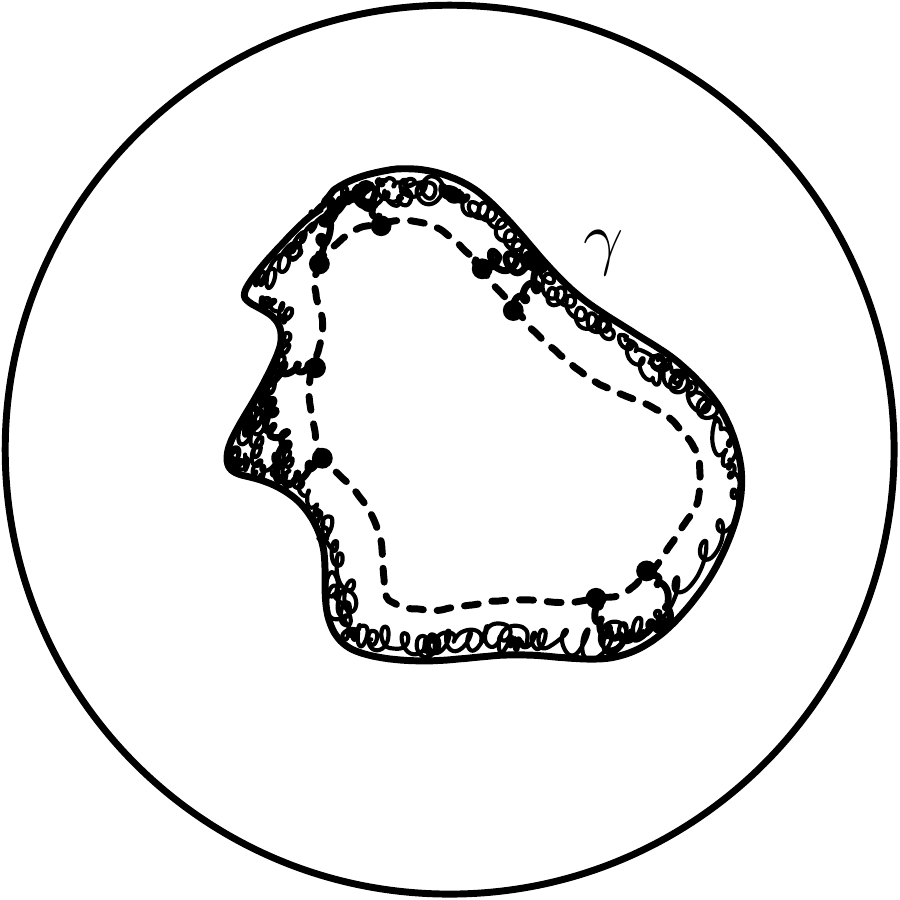}
\includegraphics[scale=0.60]{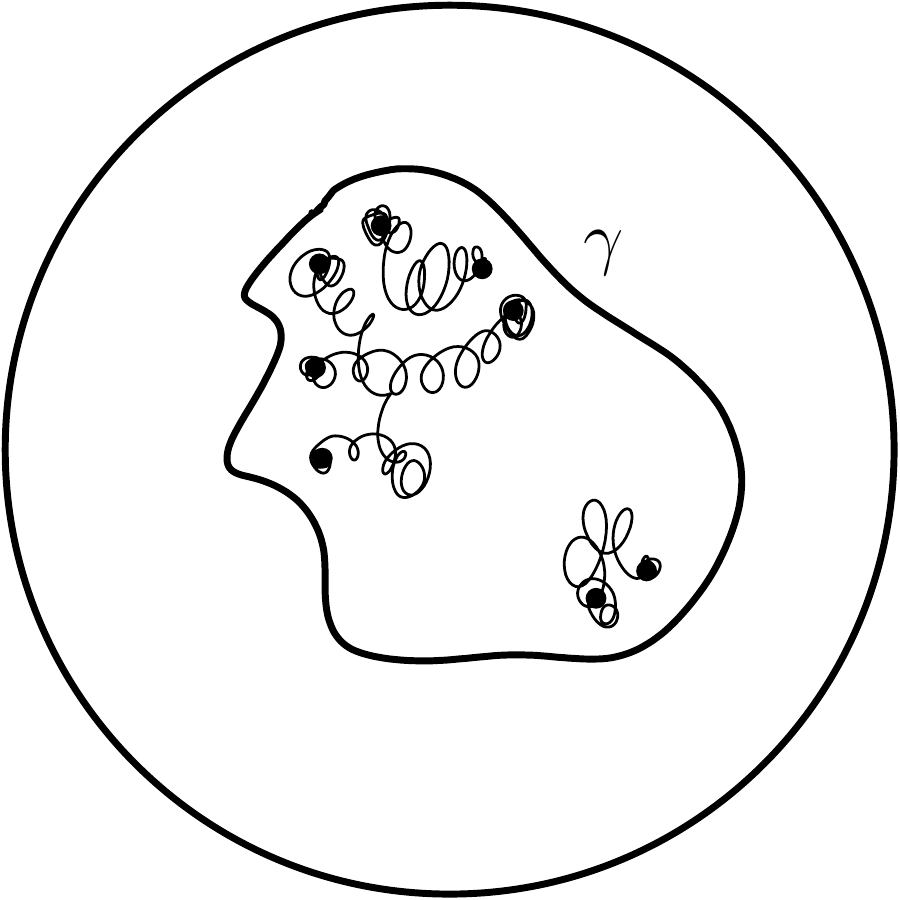}
\end{center}
\caption{\label{beads}(a) Discovering first $\gamma$ and the four beads. (b) Then sampling the remaining bridges}
\end{figure}

This approach actually suggests the following generalization of our first theorem to the multiply connected case. For instance, the doubly connected description would go as follows. 
Consider a $c=1$ Brownian loop-soup in the unit disc $\U$, and suppose that we explore the nested clusters that surround the origin, from the outside to the inside, and stop when discovering the $N$-th nested outer boundary that we call $\eta_1$ ($N$ might be random, as long as the decision when to stop the discovery is done in a Markovian way). We now fix some very small $\eps$ and explore the loop-soup clusters that surround the disc of radius $\eps$ from inside to outside, until we discover the $N'$-th one (again, $N'$ might be random), and we then call $\eta_2$ its inner boundary (mind that this $\eta_2$ might not exist, for instance if no cluster surrounds the origin). Then, the question is, on the event where $\eta_1$ surrounds $\eta_2$, how to describe the loop-soup in the annular region in-between $\eta_1$ and $\eta_2$.

Then, the arguments leading to the appropriate generalization of the first two items of Theorem \ref {mainthm} can be generalized fairly easily. As for the description of the $c=1$ excursion decomposition generalization, the observations in 
\cite {Wresampling} indicate that it should go as follows: The conditional distribution of the union of the excursions away from $\eta_1 \cup \eta_2$ of the loops that touch one or both of these loops is that of a  Poisson point process of excursions of intensity $1/4$ away from the boundary of this annular region, but conditioned   by the event that the  number of excursions that have exactly one end-point on each of $\eta_1$ and $\eta_2$ is even. 

Note that in order for $\eta_2$ to be the inner boundary of the connected component that $\eta_1$ is the outer boundary of, either there exist  excursions (and therefore at least one loop) that touch both, or there is none, but the loops of the loop-soup in the annular region between $\eta_1$ and $\eta_2$ do connect one excursion away from $\eta_1$ to one excursion away from $\eta_2$ (this is when $\eta_1$ and $\eta_2$ are only connected by a chain of loops).

\subsection {A few questions}
Our results provide some clarification about the link between the couplings between the GFF, CLE$_4$, loop-soups and their decompositions, but they do not provide answers to  all natural questions about what 
information the various constructions do provide. Let us list a few of them (we plan to address some of these in upcoming work):
\begin {itemize}
 \item We have proved that the trace (and occupation time measures) of the union of the  Brownian excursions in the Poisson point process do coincide in distribution with that of the excursions away from the outer boundary of clusters by the loops in the loop soup. It is of course natural to expect that one can actually say that the excursions away from the boundary of clusters by the loops in the loop-soup form a Poisson point process of excursions. We plan to 
 derive this in upcoming work, 
 using the previously mentioned Markovian resampling ideas in \cite {Wresampling}.
 \item Suppose that one sees all excursions away from the outer boundary of a loop-soup cluster made by the loops that intersect this boundary. What additional randomness is required in order to decide how to glue the excursions together in order to recover all the loops? Could it be that just one fair coin-toss is needed in order to decide all these connections at once? 
 \item Is the non-labelled simple non-nested CLE$_4$ in the previously described coupling a deterministic function of $\phii$? In other words, can one recover the loop-soup cluster outermost boundaries by just observing the field $\phii$? Another related question is to describe all ways to couple two Gaussian free fields in such a way that their squares are identical.  
 \item What is the conditional distribution of the set of loops that touch the boundary of the clusters, conditionally on the squared GFF in the cluster (clearly, it is not deterministic because of the resampling issues)? In other words, how does one ``separate randomly'' the excursions and the soup, when one observes their union?
\end {itemize}

\bigbreak
\noindent
{\bf Acknowledgements.}
We acknowledge support of the SNF grant SNF-155922 and of the Clay foundation, as well as the hospitality of the Isaac Newton Institute in Cambridge where part of the present work has been carried out. The authors are also part of the NCCR Swissmap. 
We also thank  David Wilson for inspiring discussions, and the referees for their comments.

\end{document}